\documentclass[11pt]{amsart}

\usepackage{amsmath}
\usepackage{amsfonts}
\usepackage{amssymb}
\usepackage{fullpage}

\usepackage{hyperref}
\hypersetup{colorlinks,linkcolor={red},citecolor={blue},urlcolor={blue}}

\newtheorem{theorem}{Theorem}[section]
\newtheorem{lemma}[theorem]{Lemma}
\newtheorem{proposition}[theorem]{Proposition}
\newtheorem{corollary}[theorem]{Corollary}
\newtheorem{conjecture}[theorem]{Conjecture}

\theoremstyle{definition}
\newtheorem{definition}[theorem]{Definition}

\newtheorem{remark}[theorem]{Remark}

\numberwithin{equation}{section}

\newcommand{\CC}{\mathbb C}
\newcommand{\HH}{\mathbb H}

\newcommand{\NN}{\mathbb N}
\newcommand{\PP}{\mathbb P}
\newcommand{\QQ}{\mathbb Q}
\newcommand{\RR}{\mathbb R}
\newcommand{\ZZ}{\mathbb Z}

\newcommand{\cD}{\mathcal D}
\newcommand{\cH}{\mathcal H}

\newcommand{\SL}{\mathop{\mathrm {SL}}\nolimits}

\newcommand{\Orth}{\mathop{\null\mathrm {O}}\nolimits}

\newcommand{\sign}{\mathop{\mathrm {sign}}\nolimits}

\newcommand{\Mp}{\mathop{\mathrm {Mp}}\nolimits}
\newcommand{\rank}{\mathop{\mathrm {rank}}\nolimits}

\newcommand{\latt}[1]{{\langle{#1}\rangle}}

\newcommand{\Grit}{\operatorname{Grit}}
\newcommand{\Borch}{\operatorname{Borch}}
\def\div{\operatorname{div}}
\newcommand{\m}{\operatorname{mod}}
\newcommand{\im}{\operatorname{Im}}

\newcommand{\mult}{\operatorname{mult}}
\newcommand{\norm}{\operatorname{Norm}}
\newcommand{\sing}{\operatorname{sing}}

\newcommand{\II}{\operatorname{II}}
\newcommand{\abs}[1]{\lvert#1\rvert}

\begin{document}

\title[The classification of 2-reflective modular forms]{The classification of 2-reflective modular forms}

\author{Haowu Wang}

\address{Center for Geometry and Physics, Institute for Basic Science (IBS), Pohang 37673, Korea}

\email{haowu.wangmath@gmail.com}


\date{\today}


\begin{abstract}
The classification of reflective modular forms is an important problem in the theory of automorphic forms on orthogonal groups. In this paper, we develop an approach based on the theory of Jacobi forms to give a full classification of 2-reflective modular forms. We prove that there are only 3 lattices of signature $(2,n)$ having 2-reflective modular forms when $n\geq 14$. We show that there are exactly 51 lattices of type $2U\oplus L(-1)$ which admit 2-reflective modular forms and satisfy that $L$ has 2-roots. We further determine all 2-reflective modular forms giving arithmetic hyperbolic 2-reflection groups.  This is the first attempt to classify reflective modular forms on lattices of arbitrary level. 
\end{abstract}

\maketitle

\section{Introduction}
A reflective modular form is a holomorphic modular form on an orthogonal group of signature $(2,n)$ whose divisor is a union of rational quadratic divisors associated to some roots of the lattice. The 2-reflective modular forms are the simplest reflective modular forms whose divisor are determined by vectors of norm $-2$, and they have the geometric interpretation as automorphic discriminants of moduli of K3 surfaces (see \cite{Nik96, BKPS98, GN17}).  Reflective modular forms are usually Borcherds products of some vector-valued modular forms (see \cite{Bru02, Bru14}). The Igusa form $\Delta_{10}$, namely the first cusp form for the Siegel modular group of genus $2$, is the first reflective modular form (see \cite{GN97}). The Borcherds form $\Phi_{12}$ for $\II_{2,26}$, i.e. the even unimodular lattice of signature $(2,26)$, is the last reflective modular form (see \cite{Bor95}).
 
Reflective modular forms are of great importance. Such modular forms play a vital role in classifying interesting Lorentzian Kac--Moody algebras, as their denominator identities are usually reflective modular forms (see  \cite{GN98a, GN98b, GN02, GN18, Sch04, Sch06, Bar03}). This type of modular forms also has applications in algebraic geometry, as the existence of a particular reflective modular form  determines the Kodaira dimension of the corresponding modular variety (see \cite{GHS07, GH14, Ma18}). In addition, reflective modular forms  are beneficial to the research of hyperbolic reflection groups and hyperbolic reflective lattices (see \cite{Bor00, Bar03}), as the existence of a reflective modular form with a Weyl vector of positive norm implies that the hyperbolic lattice is reflective (see \cite{Bor98}). This means that the subgroup generated by reflections has finite index in the integral orthogonal group of the lattice. Recently, in joint work with Gritsenko, we use the pull-backs of certain reflective modular forms of singular weight to build infinite families of remarkable Siegel paramodular forms of weights $2$ and $3$ (see \cite{GW17, GW18, GW19}). Besides, the first Fourier--Jacobi coefficients of reflective modular forms give interesting holomorphic Jacobi forms as theta blocks (see \cite{GSZ18, Gri18}). 

The classification of reflective modular forms is an open problem since 1998 when Gritsenko and Nikulin \cite{GN98a} first conjectured that the number of lattices having reflective modular forms is finite up to scaling. In the past two decades, some progress has been made on this problem. Borcherds \cite{Bor00} constructed many interesting reflective modular forms related to extraordinary hyperbolic groups as Borcherds products of weakly holomorphic modular forms on congruence subgroups. Gritsenko and Nikulin \cite{GN02} classified reflective modular forms of signature $(2,3)$ by means of the classification of hyperbolic reflective lattices. Scheithauer classified some special reflective modular forms with norm 0 Weyl vectors. More precisely, based on the theory of vector-valued modular forms, he found a necessary condition for the existence of a reflective form in \cite{Sch06}.  Using this condition, the classification of strongly reflective modular forms of singular weight (i.e. minimal weight $n/2-1$) on lattices of squarefree level is almost completed (see \cite{Sch06, Sch17, Dit18}). From an algebraic geometry approach, Ma derived the finiteness of lattices admitting  $2$-reflective modular forms and reflective modular forms of bounded vanishing order, which partly proved the conjecture of Gritsenko and Nikulin (see \cite{Ma17, Ma18}).

Scheithauer's condition is very hard to use when the lattice is not of squarefree level because in this case the Fourier coefficients of vector-valued Eisenstein series are very complicated and it is difficult to characterize the discriminant form of the lattice.  Ma's approach is ineffective to give the list of reflective lattices because his estimate is rather rough. There is no effective way to classify reflective modular forms on general lattices.  The purpose of this paper is to give a novel way to classify 2-reflective modular forms on lattices of arbitrary level.

Our method is based on the theory of Jacobi forms of lattice index (see \cite{EZ85, CG13}). We know from \cite{Bru14} that every reflective modular form on a lattice of type $U\oplus U(m)\oplus L$ is a Borcherds product of a suitable vector-valued modular form. Thus the existence of a reflective modular form is determined by the existence of a certain vector-valued modular form. In view of the isomorphism between vector-valued modular forms and Jacobi forms, we can use Jacobi forms to study reflective modular forms. In some sense,  Jacobi forms are more powerful than vector-valued modular forms. We can take the product and tensor product of different Jacobi forms. We can also consider pull-backs of Jacobi forms from a certain lattice to its sublattices. There are the Hecke type operators to raise the index of Jacobi forms and the differential operators to raise the weight of Jacobi forms. The structure of the space of Jacobi forms for some familiar lattices was known (see \cite{Wir92} for the case of root systems). Besides, we usually focus on the genus of a lattice when we use vector-valued modular forms. But we will see all the faces of a reflective modular form when we work with Jacobi forms, because there are different Jacobi forms on the expansions of an orthogonal modular form at different one-dimensional cusps. For example, the Borcherds form $\Phi_{12}$, which defines the denominator function of the fake monster algebra (see \cite{Bor90}), is constructed as the Borcherds product of $\Delta^{-1}$, where $\Delta$ is the Ramanujan Delta function
\begin{equation}\label{eq:Delta}
\Delta(\tau) = q\prod_{n=1}^\infty (1-q^n)^{24}, \quad q=e^{2\pi i\tau}, \; \tau \in \HH.  \end{equation}
But in the context of Jacobi forms, there are $24$ different constructions of this modular form corresponding to $24$ classes of positive-definite even unimodular lattices of rank $24$ (see \cite{Gri18}).

In our previous work \cite{Wan18}, we proved the nonexistence of $2$-reflective and reflective modular forms on lattices of large rank by constructing certain holomorphic Jacobi forms of small weights using differential operators. In particular, we showed that the only $2$-reflective lattices of signature $(2,n)$ satisfying $n\geq 15$ and $n\neq 19$ are $\II_{2,18}$ and $\II_{2,26}$. Here, a lattice having a $2$-reflective modular form is called $2$-reflective. In this paper, we prove the following stronger result.

\begin{theorem}\label{th:non12}
Let $M$ be a $2$-reflective lattice of signature $(2,n)$ with $n\geq 14$. Then it is isomorphic to $\II_{2,18}$, or $2U\oplus 2E_8(-1)\oplus A_1(-1)$, or $\II_{2,26}$.
\end{theorem}

We have mentioned that there is a relation between hyperbolic $2$-reflective lattices and $2$-reflective modular forms. The full classification of hyperbolic $2$-reflective lattices was known due to the work of Nikulin and Vinberg \cite{Nik81, Nik84, Vin07}. Vinberg \cite{Vin72} proved that if $U\oplus L(-1)$ is a hyperbolic $2$-reflective lattice then the set of 2-roots of each lattice in the genus of $L$ generates the whole space $L\otimes \RR$.
In this paper, we prove an analogue of Vinberg's result (see Theorem \ref{th:2-reflective}) and use it to give a complete classification of 2-reflective lattices.

\begin{theorem}\label{th:main2reflective}
There are only three types of $2$-reflective lattices containing two hyperbolic planes:
\begin{itemize}
\item[(a)] $\II_{2,26}$;
\item[(b)] $2U\oplus L(-1):$ every lattice in the genus of $L$ has no $2$-roots. In this case, every $2$-reflective modular form has a Weyl vector of norm zero and has weight $12\beta_0$, where $\beta_0$ is the multiplicity of the principal Heegner divisor $\cH_0$ defined by \eqref{eq:H0};
\item[(c)] $2U\oplus L(-1):$  every lattice in the genus of $L$ has $2$-roots and the $2$-roots generate a sublattice of the same rank as $L$. In this case, $L$ is in the genus of one of the following $50$ lattices
\[
\begin{array}{l|l}
n & L \\
\hline 
3 &A_1\\
4 &2A_1, \quad A_2\\
5 &3A_1,\quad A_1\oplus A_2,\quad A_3\\
6 &4A_1,\quad 2A_1\oplus A_2,\quad A_1\oplus A_3,\quad A_4,\quad D_4,\quad 2A_2\\
7 &5A_1,\quad 2A_1\oplus A_3,\quad A_1\oplus 2A_2,\quad A_1\oplus A_4,\quad A_1\oplus D_4,\quad A_5,\quad D_5\\
8 &6A_1,\quad 2A_1\oplus D_4,\quad A_1\oplus A_5,\quad A_1\oplus D_5,\quad E_6, \quad 3A_2,  \quad 2A_3,\quad  A_6,\quad D_6\\
9 &7A_1,\quad 3A_1\oplus D_4,\quad A_1\oplus D_6,\quad A_1\oplus E_6,\quad E_7,\quad A_7,\quad D_7\\
10 &8A_1,\quad 4A_1\oplus D_4,\quad 2A_1\oplus D_6,\quad A_1\oplus E_7,\quad E_8,\quad 2D_4,\quad D_8,\quad N_8\\
11 &5A_1\oplus D_4,\quad A_1\oplus 2D_4,\quad A_1\oplus D_8,\quad A_1\oplus E_8\\
12 &2A_1\oplus E_8\\
18 &2E_8\\
19 &2E_8\oplus A_1
\end{array} 
\]
Note that $5A_1\oplus D_4$ and $A_1\oplus N_8$, $A_1\oplus 2D_4$ and $3A_1\oplus D_6$, $A_1\oplus D_8$ and $2A_1\oplus E_7$,  $2A_1\oplus E_8$ and $D_{10}$ are in the same genus, respectively.  Here, $N_8$ is the Nikulin lattice defined as \eqref{eq:Nikulin lattice}.
Moreover, every lattice has a $2$-reflective modular form with a positive norm Weyl vector. Thus, every associated Lorentzian lattice $U\oplus L(-1)$ is hyperbolic $2$-reflective.
\end{itemize}
\end{theorem}

The above (c) characterizes the 2-reflective lattices giving arithmetic hyperbolic 2-reflection groups.  By this characterization, we further prove that there are exactly 18 hyperbolic $2$-reflective lattices of rank larger than 5 not associated to 2-reflective modular forms (see Theorem \ref{th:autocorrection}), which gives a negative answer to \cite[Problem 16.1]{Bor98}; for example $S=U\oplus E_8\oplus E_7$ is hyperbolic 2-reflective, but $U\oplus S$ is not $2$-reflective. It now remains to classify 2-reflective lattices of type (b). We conjecture that this type of lattices might come from sublattices of the Leech lattice. It seems very difficult to classify such lattices, since they correspond to hyperbolic parabolically 2-reflective lattices (see \cite{Bor00, GN18}) whose full classification is unknown. 

As a corollary of the above theorems, we figure out the classification of $2$-reflective modular forms of singular weight.  

\begin{corollary}\label{th:2reflectivesingularweight}
If $M=2U\oplus L(-1)$ has a $2$-reflective modular form of singular weight, then $L$ is in the genus of $3E_8$ or $4A_1$. 
\end{corollary}

We now explain the proof of Theorems \ref{th:non12} and \ref{th:main2reflective}. Our proof is based on manipulation of Jacobi forms and independent of the work of Nikulin and Vinberg on the classification of hyperbolic 2-reflective lattices. Suppose that $M=2U\oplus L(-1)$ and $F$ is a $2$-reflective modular form on $M$. The existence of $F$ implies that there exists a weakly holomorphic Jacobi form $\phi_{0,L}$ of weight $0$ and index $L$. The divisor of $F$ determines the singular Fourier coefficients of $\phi_{0,L}$. The singular Fourier coefficients are its Fourier coefficients of type $f(n,\ell)$ with negative hyperbolic norm $2n-(\ell,\ell)<0$. The Jacobi form $\phi_{0,L}$ has two types of singular Fourier coefficients with hyperbolic norms $-2$ and $-1/2$ respectively. We observe that all Fourier coefficients in $q^{-1}$ and $q^0$-terms of $\phi_{0,L}$ are singular except the constant term $f(0,0)$ giving the weight of $F$. The excellent thing is that the coefficients in $q^n$-terms ($n\leq 0$) of any Jacobi form of weight $0$ satisfy the following relations (see Lemma \ref{Lem:q^0-term})
\begin{align*}
C:= \frac{1}{24}\sum_{\ell\in L^\vee}f(0,\ell)-\sum_{n<0}\sum_{\ell\in L^\vee}f(n,\ell)\sigma_1(-n)&=\frac{1}{2\rank(L)} \sum_{\ell\in L^\vee}f(0,\ell)(\ell,\ell),\\
\sum_{\ell\in L^\vee}f(0,\ell)(\ell,\mathfrak{z})^2&=2C(\mathfrak{z},\mathfrak{z}),\quad \forall \ \mathfrak{z}\in L\otimes \CC.
\end{align*}
From the first identity, we deduce a formula to express the weight of $F$ in terms of the multiplicities of the irreducible components of the divisor of $F$. From the second identity, we derive that if $L$ has $2$-roots then the set of all $2$-roots spans the whole space $L\otimes \RR$. Moreover, all irreducible root components not of type $A_1$ have the same Coxeter number (see Theorem \ref{th:2-reflective}). By virtue of these results, we only need to consider a finite number of lattices.
Furthermore, the $q^0$-term of $F$ also defines a holomorphic Jacobi form as a theta block (see \eqref{FJtheta}).
From its holomorphicity, we also deduce a necessary condition. 
The $2$-reflective modular forms on lattices listed in assertion (c) can be constructed as quasi pull-backs of the Borcherds form $\Phi_{12}$ (see \S \ref{sec:quasipullback}).
For other lattices, the quasi pull-backs of $\Phi_{12}$ are not exactly 2-reflective modular forms and usually have additional divisors. But it is not bad. By considering the difference between the pull-back and the assumed 2-reflective modular form, we construct some Jacobi forms whose nonexistence can be proved by the structure of the space of Jacobi forms. Combining these arguments together, the theorems can be proved. 

\begin{remark}
The approach to prove Theorem \ref{th:main2reflective} was later greatly used in our subsequent papers to classify arithmetic groups acting on symmetric domains of type IV for which the ring of modular forms is freely generated (see [Compos. Math. 157 (2021), pp. 2026--2045]), and to classify reflective lattices of prime level (see [Trans. Amer. Math. Soc. 375 (2022), pp. 3451--3468]). 
\end{remark}

The paper is organized as follows. In \S \ref{Sec:lattices} we recall the necessary material on lattices and discriminant forms. In \S \ref{Sec:reflective} we give the definitions of reflective modular forms. In \S \ref{Sec:Jacobi} we introduce the theory of Jacobi forms of lattice index. In \S \ref{sec:quasipullback} we show how to use quasi pull-backs of the Borcherds form $\Phi_{12}$ to construct reflective modular forms. \S \ref{Sec:classification2} is the heart of this paper. We prove the main theorems and some other classification results. In \S \ref{sec:autocorrection} we consider the automorphic corrections of hyperbolic 2-reflective lattices. In \S \ref{sec:K3surfaces} we prove that the lattice related to the moduli space of polarized K3 surfaces $2U\oplus 2E_8(-1)\oplus \latt{-2n}$ is reflective if and only if $n=1$, $2$ (see Theorem \ref{th:reflective19}). This generalizes one result in \cite{Loo03}. In \S \ref{sec:open} we give some remarks and formulate five open questions related to this paper.

\section{Lattices and discriminant forms}\label{Sec:lattices}
In this section we recall some basic results on lattices and discriminant forms. The main references for this material are \cite{Bou60,Ebe02,Nik80,SC98}.  

Let $M$ be a lattice equipped with a non-degenerate integral valued symmetric bilinear form $(\cdot,\cdot)$ with even valued norm. Such $M$ is called an even lattice. The associated quadratic form is defined as $Q(x)=(x,x)/2$. Let $M^\vee$ denote the dual lattice of $M$ and $\rank(M)$ denote the rank of $M$.  For every $a\in \ZZ\backslash\{0\}$, the lattice obtained by rescaling $M$ with $a$ is denoted by $M(a)$. It is endowed with the quadratic form $a\cdot Q$ instead of $Q$. The level of $M$ is the smallest positive integer $N$ such that $NQ(x)\in \ZZ$ for all $x\in M^\vee$. If $M$ is of level $N$ then $NM^\vee \subseteq M$.  If $x\in M$ satisfying $\QQ x\cap M=\ZZ x$, then it is called primitive. For any non-zero $x\in M$ the divisor of $x$ is the natural number $\div(x)$ defined by $(x,M)=\div(x)\ZZ$. Note that $x/\div(x)$ is a primitive element in $M^\vee$.  An embedding $M_1 \hookrightarrow M_2$ of even lattices is called primitive if $M_2/M_1$ is a free $\ZZ$-module. A given embedding $M \hookrightarrow M_1$ of even lattices, for which $M_1/M$ is a finite abelian group, is called an even overlattice of $M$.

A finite abelian group $D$ equipped with a non-degenerate quadratic form $Q: D \to \QQ/ \ZZ$ is called a discriminant form. The residue class $\sign(D) \in \ZZ/8\ZZ$ defined by Milgram’s formula 
$$
\sum_{\gamma\in D} \exp(2\pi iQ(\gamma)) = \sqrt{|D|}\exp(2\pi i\sign(D)/8)
$$
is called the signature of $D$.
Obviously, the discriminant group $D(M)=M^\vee/M$ with the induced quadratic form $Q$ is a discriminant form. A subgroup $G$ of $D(M)$ is called isotropic if $Q(\gamma)=0$ for any $\gamma\in G$. There is a one-to-one correspondence between even overlattices of $M$ and isotropic subgroups of $D(M)$. On the one hand, if $M_1$ is an even overlattice of $M$, then $M_1/M$ is an isotropic subgroup of $D(M)$. On the other hand, if $G$ is an isotropic subgroup of $D(M)$, then the lattice generated by $G$ over $M$ is an even overlattice of $M$.

A suitable notion to classify even lattices is that of genus. The genus of a lattice $M$ is the set of lattices $M'$ of the same signature as $M$ such that $M\otimes \ZZ_p \cong M'\otimes \ZZ_p$ for every prime number $p$. By \cite{Nik80}, two even lattices of the same signature are in the same genus if and only if their discriminant forms are isomorphic.  Thus we here use the following equivalent definition of genus. Let $M$ be an even lattice of signature $(r,s)$ with discriminant form $D$. The genus of $M$, which is denoted by $\II_{r,s}(D)$, is the set of all even lattices of signature $(r,s)$ whose discriminant form is isomorphic to $D$. A discriminant form can decompose into a sum of indecomposable Jordan components (see \cite{SC98}, \cite{Ber00}, or \cite{Sch06} for details). We denote the even unimodular lattice of signature $(2,n)$ by $\II_{2,n}$. We state the following theorems proved in \cite{Nik80}, which tell us when a given genus is non-empty and when a given genus contains only one lattice up to isomorphism. 

\begin{theorem}[Corollary 1.10.2 in \cite{Nik80}]\label{th: Nik1}
Let $D$ be a discriminant form and $r,s\in \ZZ$. If $r\geq 0$, $s\geq 0$, $r-s = \sign(D) \m 8$ and $r+s> l(D)$, then there is an even lattice of signature $(r,s)$ having discriminant form $D$. Here, $l(D)$ is the minimum number of generators of the group $D$.
\end{theorem}

\begin{theorem}[Corollary 1.13.3 in \cite{Nik80}]\label{th: Nik2}
Let $D$ be a discriminant form and $r,s\in \ZZ$. If $r\geq 1$, $s\geq 1$ and $r+s\geq 2+l(D)$, then all even lattices of genus $\II_{r,s}(D)$ are isomorphic. 
\end{theorem}

Let $U$ be a hyperbolic plane i.e. $U=\ZZ e+\ZZ f$ with $(e,e)=(f,f)=0$ and $(e,f)=1$. The lattice $U$ is an even unimodular lattice of signature $(1,1)$. As a consequence of Theorems \ref{th: Nik1} and \ref{th: Nik2}, we prove the following criterion.

\begin{lemma}\label{lem:2U}
Let $M$ be an even lattice of signature $(2,n)$ with $n\geq 3$. If the minimum number of generators of $D(M)$ satisfies $n-2>l(D(M))$, then there exists a negative-definite even lattice $L$ such that $M=2U\oplus L$.
\end{lemma}

\begin{proof}
By Theorem \ref{th: Nik1}, there exists a negative-definite even lattice $L$ of rank $n-2$ whose discriminant form is isomorphic to $D(M)$. By Theorem \ref{th: Nik2}, $2U\oplus L$ is isomorphic to $M$. 
\end{proof}

At the end of this section, we recall some basic facts on root lattices following \cite{Ebe02}.  Let $L$ be an even lattice in $\RR^N$. An element $r \in L$ is called a 2-root if $(r, r) = 2$. The
set of all 2-roots is denoted by $R_L$. The lattice $L$ is called a root lattice if $L$ is generated by $R_L$. Every root lattice can be written as an orthogonal direct sum of the irreducible root lattices of types $A_n (n\geq 1)$, $D_n (n\geq 4)$, $E_6$, $E_7$, and $E_8$. For a root lattice $L$ of rank $n$, the number $h=\abs{R_L}/n$ is called the Coxeter number of $L$. The Coxeter numbers of the irreducible root lattices are listed in Table \ref{tab:Coxeter}.

\begin{table}[ht]
\caption{Coxeter numbers of irreducible root lattices}
\label{tab:Coxeter}
\renewcommand\arraystretch{1.5}
\noindent\[
\begin{array}{|c|c|c|c|c|c|}
\hline 
\text{L} & A_n & D_n & E_6 & E_7 & E_8 \\ 
\hline 
h & n+1 & 2(n-1) & 12& 18& 30 \\ 
\hline 
\end{array} 
\]
\end{table}

By \cite[Proposition 1.6]{Ebe02}, we have the following identity.
\begin{proposition}\label{prop:root}
Let $L\subsetneq \RR^n$ be an irreducible root lattice. Then for any $x\in \RR^n$ we have
$$
\sum_{r\in R_L} (r,x)^2= 2h(x,x).
$$ 
\end{proposition}
Let $R_L^+$ be the set of positive roots of $L$. The Weyl vector of $L$ is defined as $\rho=\frac{1}{2}\sum_{r\in R_L^+} r$.  We know from \cite[Lemma 1.16]{Ebe02} that the norm of the Weyl vector of an irreducible root lattice is given by $\rho^2=\frac{1}{12}h(h+1)\rank(L)$.

\section{Reflective modular forms}\label{Sec:reflective}
In this section we introduce the definition and some basic properties of reflective modular forms.

Let $M$ be an even lattice of signature $(2,n)$ with $n\geq 3$. Its associated Hermitian symmetric domain of type IV has two connected components and we fix one of them 
\begin{equation*}
\cD(M)=\{[\mathcal{Z}] \in  \PP(M\otimes \CC):  (\mathcal{Z}, \mathcal{Z})=0, (\mathcal{Z},\mathcal{Z}) > 0\}^{+}.
\end{equation*}
Let $\Orth^+(M) \subset \Orth(M)$ denote the index $2$ subgroup preserving $\cD(M)$. The stable orthogonal group $\widetilde{\Orth}^+(M)$ is a subgroup of $\Orth^+ (M)$ acting trivially on the discriminant form $D(M)$. Let $\Gamma$ be a finite index subgroup of $\Orth^+ (M)$ and $k\in \ZZ$. A modular form of weight $k$ and character $\chi: \Gamma\to \CC^*$ with respect to $\Gamma$ is a holomorphic function $F: \cD(M)^{\bullet}\to \CC$ on the affine cone $\cD(M)^{\bullet}$ satisfying
\begin{align*}
F(t\mathcal{Z})&=t^{-k}F(\mathcal{Z}), \quad \forall t \in \CC^*,\\
F(g\mathcal{Z})&=\chi(g)F(\mathcal{Z}), \quad \forall g\in \Gamma.
\end{align*}
A modular form is called a cusp form if it vanishes at every cusp (i.e. a boundary component of the Baily--Borel compactification of the modular variety $\Gamma\backslash \cD(M)$).

The non-zero modular form $F$ either has weight 0 in which case it is constant, or has weight at least $n/2-1$. The minimal possible positive weight $n/2-1$ is called the singular weight.

For any $v\in M^\vee$ of negative norm, the rational quadratic divisor associated to $v$ is defined as
\begin{equation}
 \cD_v(M)=v^\perp\cap \cD(M)=\{ [\mathcal{Z}]\in \cD(M) : (\mathcal{Z},v)=0\}. 
\end{equation}
The reflection with respect to the hyperplane defined by an anisotropic vector $r$ is
\begin{equation}
\sigma_r(x)=x-\frac{2(r,x)}{(r,r)}r,  \quad x\in M.
\end{equation}
A primitive vector $l\in M$ of negative norm is called reflective if the reflection $\sigma_l$ is in $\Orth^+(M)$. The divisor $\cD_v(M)$ is called reflective if $\sigma_v\in\Orth^+(M)$. For $\lambda \in D(M)$ and $m\in \QQ$, we define
\begin{equation}
\cH(\lambda,m)=\bigcup_{\substack{v \in M+\lambda \\ (v,v)=2m}}  \cD_v(M)
\end{equation}
as the Heegner divisor of discriminant $(\lambda,m)$. 

Remark that a primitive vector $l\in M$ with $(l,l)=-2d$ is reflective if and only if $\div(l)=2d$ or $d$.  We set $\lambda=[l/\div(l)]\in D(M)$. Then $\cD_l(M)$ is contained in $\cH(\lambda,-1/(4d))$ if $\div(l)=2d$, and is contained in 
$$\cH(\lambda,-1/d)-\sum_{2\nu=\lambda}\cH(\nu,-1/(4d))$$
if $\div(l)=d$.  In particular, when $M$ is of prime level $p$, a primitive vector $l\in M$ is reflective if and only if $(l,l)=-2$ and $\div(l)=1$, or $(l,l)=-2p$ and $\div(l)=p$.

\begin{definition}
Let $F$ be a non-constant holomorphic modular form on $\cD(M)$ with respect to a finite index subgroup $\Gamma < \Orth^+ (M)$ and a character $\chi$. The function $F$ is called reflective if the support of its zero divisor is contained in the union of reflective divisors.  A lattice $M$ is called reflective if it admits a reflective modular form.
\end{definition}

\begin{definition}
A non-constant holomorphic modular form on $\cD(M)$ is called 2-reflective if the support of its zero divisor is contained in the Heegner divisor generated by $(-2)$-vectors in $M$ i.e.
\begin{equation}\label{Heegner}
\cH =\cH(0,-1)=\bigcup_{\substack{v \in M\\ (v,v)=-2}} \cD_v(M).
\end{equation}
A $2$-reflective modular form $F$ is called a modular form with complete 2-divisor if $\div(F)=\cH$. A lattice $M$ is called 2-reflective if it admits a $2$-reflective modular form. 
\end{definition}

All $(-2)$-vectors are reflective because the reflections associated to $(-2)$-vectors are in $\widetilde{\Orth}^+(M)$. Therefore, $2$-reflective modular forms are special reflective modular forms. 

As in \cite[Lemma 2.2]{Ma17}, we can show that if $M$ admits a reflective (resp. $2$-reflective) modular form with respect to some $\Gamma < \Orth^+ (M)$ then $M$ also has a reflective (resp. $2$-reflective) modular form with respect to any other finite index subgroup $\Gamma' < \Orth^+ (M)$. Therefore, throughout this paper,  we only consider reflective (resp. 2-reflective) modular forms with respect to $\widetilde{\Orth}^+(M)$.

The following lemma is useful to classify 2-reflective lattices. 

\begin{lemma}[Lemma 2.3 in \cite{Ma17}]\label{Lem:reductionMa}
If $M$ is $2$-reflective, then any even overlattice $M'$ of $M$ is also $2$-reflective.  If $M$ is not $2$-reflective, neither is any finite-index sublattice
of $M$.
\end{lemma} 

Remark that Lemma \ref{Lem:reductionMa} does not hold for reflective modular forms because $\Orth^+(M)$ is not contained in $\Orth^+(M')$ in general and a reflective divisor $\cD_v$ in $\cD(M)$ is usually not a reflective divisor in $\cD(M')$.

\section{Jacobi forms of lattice index}\label{Sec:Jacobi}
In this section, we briefly introduce the theory of Jacobi forms of lattice index. We refer to \cite{CG13} for more details. Let $L$ be an even positive-definite lattice with bilinear form $(\cdot, \cdot)$ and dual lattice $L^\vee$.  Let $M=U\oplus U_1\oplus L(-1)$,
where $U=\ZZ e\oplus\ZZ f$ and $U_1=\ZZ e_1\oplus\ZZ f_1$ are two hyperbolic planes. We fix a basis of $M$ of the form $
(e,e_1,...,f_1,f)$, 
where $...$ denotes a basis of $L(-1)$.  The homogeneous domain $\cD(M)$ has a tube realization at the 1-dimensional cusp determined by the isotropic plane $\latt{e,e_1}$:
$$
\cH(L)=\{Z=(\tau,\mathfrak{z},\omega)\in \HH\times (L\otimes\CC)\times \HH: 
(\im Z,\im Z)_M>0\}, 
$$
where $(\im Z,\im Z)_M=2\im \tau \im \omega - 
(\im \mathfrak{z},\im \mathfrak{z})$. 
A Jacobi form can be regarded as a modular form with 
respect to the Jacobi group $\Gamma^J(L)$ which is the parabolic 
subgroup preserving the  isotropic plane $\latt{e,e_1}$ and acting trivially on $L$.
The Jacobi group is the semidirect product of $\SL_2(\ZZ)$ with the 
Heisenberg group $H(L)$ of $L$. 
The analytic definition of Jacobi forms is as follows

\begin{definition}
Let $\varphi : \HH \times (L \otimes \CC) \rightarrow \CC$ be a holomorphic function and $k\in\ZZ$. If $\varphi$ satisfies the functional equations
\begin{align*}
\varphi \left( \frac{a\tau +b}{c\tau + d},\frac{\mathfrak{z}}{c\tau + d} \right)& = (c\tau + d)^k \exp \left(i \pi \frac{c(\mathfrak{z},\mathfrak{z})}{c \tau + d}\right) \varphi ( \tau, \mathfrak{z} ), \quad \left( \begin{array}{cc}
a & b \\ 
c & d
\end{array} \right)   \in \SL_2(\ZZ)\\
\varphi (\tau, \mathfrak{z}+ x \tau + y)&= \exp \left(-i \pi ( (x,x)\tau +2(x,\mathfrak{z}) )\right) \varphi ( \tau, \mathfrak{z} ), \quad x,y \in L,
\end{align*}
and if $\varphi$ admits a Fourier expansion of the form 
\begin{equation}\label{eq:FC}
\varphi ( \tau, \mathfrak{z} )= \sum_{n\geq n_0 }\sum_{\ell\in L^\vee}f(n,\ell)q^n\zeta^\ell
\end{equation}
where $n_0$ is a constant, $q=e^{2\pi i \tau}$ and $\zeta^\ell=e^{2\pi i (\ell,\mathfrak{z})}$, then $\varphi$ is called a weakly holomorphic Jacobi form of weight $k$ and index $L$. If $\varphi$ further satisfies the condition
($ f(n,\ell) \neq 0 \Longrightarrow n \geq 0 $)
then it is called a weak Jacobi form.  If $\varphi$ further satisfies the stronger condition
($ f(n,\ell) \neq 0 \Longrightarrow 2n - (\ell,\ell) \geq 0 $)
then it is called a holomorphic Jacobi form.
We denote by 
$$
J^{!}_{k,L}\supset J^{w}_{k,L}\supset J_{k,L}
$$
the vector spaces of weakly holomorphic, weak, holomorphic Jacobi forms of weight $k$ and index $L$.
\end{definition}

Remark that the Jacobi forms for the lattice $A_1$ are actually the classical Jacobi forms due to Eichler and Zagier \cite{EZ85}. In the literature, Jacobi forms of weight $k$ and index $L(t)$ are also called Jacobi forms of weight $k$ and index $t$ for the lattice $L$, where $t$ is a positive integer.

The Fourier coefficient $f(n,\ell)$ depends only on the number $2n-(\ell,\ell)$ and the class of $\ell$ modulo $L$. Besides, $f(n,\ell)=(-1)^kf(n,-\ell)$. If $\varphi$ is a weak Jacobi form, then its Fourier coefficients satisfy 
$$ 
f(n,\ell) \neq 0 \Longrightarrow 2n- (\ell,\ell) \geq - \min\{(v,v): v\in \ell+ L \}.
$$
If $\varphi$ is a weakly holomorphic Jacobi form and $n_0<0$, then $\Delta^{-n_0}\varphi$ will be a weak Jacobi form, where $\Delta$ is defined in \eqref{eq:Delta}. Thus the above relation implies that for any $n$ the number of non-zero terms of the $\ell$ sum in \eqref{eq:FC} is finite. 
The number $2n-(\ell,\ell)$ is called the hyperbolic norm of the Fourier coefficient $f(n,\ell)$. The Fourier coefficients $f(n,\ell)$ with negative hyperbolic norm are called singular Fourier coefficients, which determine the divisor of Borcherds products.  It is clear from the definition that a weakly holomorphic Jacobi form without singular Fourier coefficients is a holomorphic Jacobi form. 

We next explain the relation between modular forms for the Weil representation and Jacobi forms.  We denote by $\{\textbf{e}_\gamma: \gamma \in D(L)\}$ the formal basis of the group ring $\CC[D(L)]$. Let $\Mp_2(\ZZ)$ be the metaplectic group which is a double cover of $\SL_2(\ZZ)$. The Weil representation of $\Mp_2(\ZZ)$ on $\CC[D(L)]$ is denoted by $\rho_{D(L)}$ (see \cite[Section 1.1]{Bru02}). Let $F$ be a vector-valued, weakly holomorphic (i.e. holomorphic except at infinity) modular form for $\rho_{D(L)}$ of weight $k$ with  Fourier expansion  
\begin{align*}
F(\tau)=\sum_{\gamma\in D(L)}\sum_{n\in\ZZ-\frac{(\gamma,\gamma)}{2}} c(\gamma,n)q^n \textbf{e}_\gamma
 =\sum_{\gamma\in D(L)} F_{\gamma}(\tau)\textbf{e}_\gamma.
\end{align*}
The Fourier coefficients $c(\gamma,n)$ with negative $n$ are called singular.

Recall that the theta-functions for the lattice $L$ are defined as
\begin{equation}\label{eq:ThetaFunction}
\Theta_{\gamma}^{L}(\tau,\mathfrak{z})=\sum_{\ell \in \gamma +L}\exp\left(\pi i(\ell,\ell) \tau + 2\pi i(\ell,\mathfrak{z}) \right), \quad \gamma\in D(L).
\end{equation}
Then the map 
\begin{equation}
F(\tau) \longmapsto \Phi(F)(\tau,\mathfrak{z})=\sum_{\gamma\in D(L)} F_{\gamma}(\tau)\Theta_{\gamma}^{L}(\tau,\mathfrak{z})
\end{equation}
defines an isomorphism between the space of weakly holomorphic modular forms of weight $k$ for $\rho_{D(L)}$ and the space of weakly holomorphic Jacobi form of weight $k+\rank (L)/2$ and index $L$. 
This map sends singular Fourier coefficients of vector-valued modular forms to singular Fourier coefficients of Jacobi forms. Therefore, it induces an isomorphism between the subspaces of holomorphic vector-valued modular forms of weight $k$ and holomorphic Jacobi forms of weight $k+\rank (L)/2$.  From this, we deduce that $J_{k,L}=\{0\}$ if $k<\rank(L)/2$. The minimum possible weight $k=\rank(L)/2$ is called the singular weight. For a holomorphic Jacobi form of singular weight, its non-zero Fourier coefficients $f(n,\ell)$ satisfy $2n-(\ell,\ell)=0$. Note that these results hold for holomorphic Jacobi forms with a character.  

The following differential operators are very useful. We refer to \cite[Lemma 2.2]{Wan18} for a proof.

\begin{lemma}\label{Lem:differential}
Let $\psi(\tau,\mathfrak{z})=\sum f(n,\ell)q^n\zeta^\ell$ be a weakly holomorphic Jacobi form of weight $k$ and index $L$. We define
\begin{align*}
H_k(\psi)&:=H(\psi)+(2k-\rank(L))G_2\psi,\\
H(\psi)(\tau,\mathfrak{z})&:=\frac{1}{2}\sum_{n\in \ZZ}\sum_{\ell\in L^\vee} \left[2n-(\ell,\ell) \right]f(n,\ell)q^n\zeta^\ell,
\end{align*}
where $G_2(\tau)=-\frac{1}{24}+\sum_{n\geq 1}\sigma_1(n)q^n$ and $\sigma_1(m)=\sum_{d \vert m} d$. 
Then $H_k(\psi)$ is a weakly holomorphic Jacobi form of weight $k+2$ and index $L$.
\end{lemma}

The next lemma gives useful identities related to singular Fourier coefficients of Jacobi forms of weight $0$, which plays a crucial role in this paper. We refer to \cite[Proposition 2.6]{Gri18} for a proof. Its variant in the context of vector-valued modular forms was first proved in \cite[Theorem 10.5]{Bor98}.

\begin{lemma}\label{Lem:q^0-term}
Let $\phi$ be a weakly holomorphic Jacobi form of weight $0$ and index $L$ with the Fourier expansion
$$
\phi(\tau,\mathfrak{z})=\sum_{n\in\ZZ}\sum_{\ell\in L^\vee}f(n,\ell)q^n\zeta^\ell.
$$ 
Then we have the following identity
\begin{equation}\label{eq:q^0-term}
C:= \frac{1}{24}\sum_{\ell\in L^\vee}f(0,\ell)-\sum_{n<0}\sum_{\ell\in L^\vee}f(n,\ell)\sigma_1(-n)=\frac{1}{2\rank(L)} \sum_{\ell\in L^\vee}f(0,\ell)(\ell,\ell).
\end{equation}
Moreover, we have
\begin{equation}\label{eq:vectorsystem}
\sum_{\ell\in L^\vee}f(0,\ell)(\ell,\mathfrak{z})^2=2C(\mathfrak{z},\mathfrak{z}),\quad \forall\ \mathfrak{z}\in L\otimes \CC.
\end{equation}
\end{lemma}

\begin{remark}
Let $\Lambda$ be an even positive-definite unimodular lattice of rank $24$. Assume that the set $R_\Lambda$ of 2-roots of $\Lambda$ is non-empty.  Let $R(\Lambda)$ denote the root lattice generated by $R_\Lambda$. The theta-function for $\Lambda$ is a holomorphic Jacobi form of weight $12$ and index $\Lambda$. Thus, we have
$$
\psi_{0,\Lambda}(\tau,\mathfrak{z})=\frac{\Theta_\Lambda(\tau,\mathfrak{z})}{\Delta(\tau)}=q^{-1}+\sum_{r\in R_\Lambda} \zeta^{r}+ 24 +O(q) \in J_{0,\Lambda}^!,
$$
where $\Theta_\Lambda=\Theta_0^\Lambda$ (see \eqref{eq:ThetaFunction}).
By Lemma \ref{Lem:q^0-term}, we prove the identity $\sum_{r\in R_\Lambda} (r,\mathfrak{z})^2= 2h (\mathfrak{z}, \mathfrak{z})$. It follows that the lattice $R(\Lambda)$ has rank $24$ and all its irreducible components have the same Coxeter number. In this paper, we shall use the similar idea to classify 2-reflective modular forms.
\end{remark}

Using Lemma \ref{Lem:q^0-term}, we give a simple proof of \cite[Theorem 11.2]{Bor98} in the context of Jacobi forms.
\begin{corollary}
Let $\phi$ be a weakly holomorphic Jacobi form of weight $0$ and index $L$ with the Fourier expansion
$$
\phi(\tau,\mathfrak{z})=\sum_{n\in\ZZ}\sum_{\ell\in L^\vee}f(n,\ell)q^n\zeta^\ell.
$$ 
Assume that $f(n,\ell)\in \ZZ$ for all $n\leq 0$ and $\ell \in L^\vee$. Let $(n(L))$ denote the ideal of $\ZZ$ generated by $(x,y)$, $x,y\in L$. Then we have 
$$
\frac{n(L)}{24}\sum_{\ell\in L^\vee}f(0,\ell)\in \ZZ.
$$
\end{corollary}
\begin{proof}
By \eqref{eq:vectorsystem}, we get
$$
\sum_{\ell\in L^\vee}f(0,\ell)(\ell,x)(\ell,y)=2C(x,y),\quad \forall x,y\in L.
$$
Since $f(0,\ell)=f(0,-\ell)\in\ZZ$, we get $C(x,y)\in \ZZ$ for all $x,y\in L$, which yields $n(L)C\in \ZZ$. We thus complete the proof by \eqref{eq:q^0-term}.
\end{proof}

We next introduce the Borcherds products. The input data of original Borcherds lifting are modular forms for the Weil representation. The constructed orthogonal modular forms have nice infinite product expansions at rational 0-dimensional cusps. By means of the isomorphism between modular forms for the Weil representation and Jacobi forms, Gritsenko and Nikulin \cite{GN98b} proposed a variant of Borcherds products, which lifts weakly holomorphic Jacobi forms of weight $0$ to modular forms on orthogonal groups.   In this case, any constructed modular form has a nice infinite product expansion at each rational 1-dimensional cusp and can be expressed as a product of a general theta block with the exponential of additive Jacobi lifting.

\begin{theorem}[Theorem 4.2 in \cite{Gri18}]\label{th:BorcherdsJF}
Let 
$$
\varphi(\tau,\mathfrak{z})=\sum_{n\in\ZZ}\sum_{\ell\in L^\vee}f(n,\ell)q^n \zeta^\ell \in J^{!}_{0,L}.
$$
Assume that $f(n,\ell)\in \ZZ$ for all $2n-(\ell,\ell)< 0$.   We set
\begin{align*}
&A=\frac{1}{24}\sum_{\ell\in L^\vee}f(0,\ell),& &\vec{B}=\frac{1}{2}\sum_{\ell>0} f(0,\ell)\ell,& &C=\frac{1}{2\rank(L)}\sum_{\ell\in L^\vee}f(0,\ell)(\ell,\ell).&
\end{align*}
Then the product
$$\Borch(\varphi)(Z)=q^A \zeta^{\vec{B}} \xi^C\prod_{\substack{n,m\in\ZZ, \ell\in L^\vee\\ (n,\ell,m)>0}}(1-q^n \zeta^\ell \xi^m)^{f(nm,\ell)}, $$
where $Z= (\tau,\mathfrak{z}, \omega) \in \cH(L)$, $q=\exp(2\pi i \tau)$, $\zeta^\ell=\exp(2\pi i (\ell, \mathfrak{z}))$, $\xi=\exp(2\pi i \omega)$, defines a meromorphic modular form of weight $f(0,0)/2$ for  $\widetilde{\Orth}^+(2U\oplus L(-1))$ with a character $\chi$ induced by
\begin{align*}
&\chi \lvert_{\SL_2(\ZZ)}=v_\eta^{24A},& &\chi \lvert_{H(L)}([\lambda,\mu; r])=e^{\pi i C((\lambda,\lambda)+(\mu, \mu )- (\lambda, \mu ) +2r)},& & \chi(V)=(-1)^D,&
\end{align*}
where $V: (\tau,\mathfrak{z}, \omega) \mapsto (\omega,\mathfrak{z},\tau)$ and $D=\sum_{n<0}\sigma_0(-n) f(n,0)$. The poles and zeros of $\Borch(\varphi)$ lie on the rational quadratic divisors $\cD_v(2U\oplus L(-1))$, where $v\in 2U\oplus L^\vee(-1)$ is a primitive vector with $(v,v)<0$. The multiplicity of this divisor is given by 
$$ \mult \cD_v(2U\oplus L(-1)) = \sum_{d\in \ZZ,d>0 } f(d^2n,d\ell),$$
where $n\in\ZZ$, $\ell\in L^\vee$ such that $(v,v)=2n-(\ell,\ell)$ and $v\equiv \ell\mod 2U\oplus L(-1)$.
Moreover, we have
$$ 
\Borch(\varphi)=\psi_{L,C}(\tau,\mathfrak{z})\xi^C \exp \left(-\Grit(\varphi) \right),$$
where $\Grit(\varphi)$ is the additive Jacobi lifting of $\varphi$ and the first Fourier--Jacobi coefficient is given by
\begin{equation}\label{FJtheta}
\psi_{L,C}(\tau,\mathfrak{z})=\eta(\tau)^{f(0,0)}\prod_{\ell >0}\left(\frac{\vartheta(\tau,(\ell,\mathfrak{z}))}{\eta(\tau)} \right)^{f(0,\ell)}.
\end{equation}
The Weyl vector of the Borcherds product is $(A,\vec{B},C)$. 
\end{theorem}

We explain some notations in the above theorem. 
The odd Jacobi theta series $\vartheta$ is defined as 
$$
\vartheta(\tau,z) =q^{\frac{1}{8}}(\zeta^{\frac{1}{2}}-\zeta^{-\frac{1}{2}}) \prod_{n\geq 1} (1-q^{n}\zeta)(1-q^n \zeta^{-1})(1-q^n), \quad \zeta=e^{2\pi iz},
$$
which is a holomorphic Jacobi form of weight $1/2$ and index $1/2$ with a multiplier system of order $8$ in the sense of Eichler--Zagier (see \cite[Example 1.5]{GN98b}). The form $\eta$ is the Dedekind Eta function
$$
\eta(\tau)=q^{1/24}\prod_{n\geq 1}(1-q^n).
$$
By \eqref{eq:vectorsystem}, the finite multiset $X=\{\ell;f(0,\ell)\}$ from Theorem \ref{th:BorcherdsJF} forms a vector system defined in \cite[\S 6]{Bor95}. We define its Weyl chamber as a connected component of
$$
L\otimes \RR \backslash \Big( \bigcup_{x\in X\backslash \{0\}} \{ v\in L\otimes \RR: (x,v)=0\}  \Big).
$$
Let $W$ be a fixed Weyl chamber. For $\ell\in L^\vee$, we define an ordering on $L^\vee$ by
$$
\ell >0 \iff \exists\; w \in W \;\text{s.t.}\; (\ell, w)>0.
$$
The notation $(n,\ell,m)>0$ in Theorem \ref{th:BorcherdsJF} means that either $m>0$, or $m=0$ and $n>0$, or $m=n=0$ and $\ell <0$. 

We emphasize the fact that if $\Borch(\varphi)$ is holomorphic then its first Fourier--Jacobi coefficient $\psi_{L,C}$ (see \eqref{FJtheta}) is a holomorphic Jacobi form of weight $f(0,0)/2$ and index $L(C)$.

\section{Quasi pull-backs of modular forms}\label{sec:quasipullback}
In this section we introduce the quasi pull-backs of modular forms and employ this technique to construct many reflective modular forms.

Borcherds \cite{Bor95} constructed a modular form of singular weight and character $\det$ on $\Orth^+(\II_{2,26})$
$$
\Phi_{12}\in M_{12}(\Orth^+(\II_{2,26}),\det),
$$
where $\II_{2,26}$ is the unique even unimodular lattice of signature $(2,26)$. The function $\Phi_{12}$ is constructed as the Borcherds product of the inverse of $\Delta$ defined by \eqref{eq:Delta}
$$
1/\Delta(\tau)=q^{-1}+24+324q+3200q^2+\cdots
$$
and it is a modular form with complete 2-divisor i.e.
$$
\div(\Phi_{12})=\cH=\sum_{\substack{v\in \II_{2,26}/\{\pm 1\}\\ (v,v)=-2}} \cD_v(\II_{2,26}).
$$
By the Eichler criterion (see \cite[Proposition 4.1]{Gri18}), all $(-2)$-vectors in $\II_{2,26}$ form only one orbit with respect to $\Orth^+(\II_{2,26})$. We next introduce the quasi pull-back of $\Phi_{12}$. 

First we give a general property of rational quadratic divisors. Let $M$ be an even lattice of signature $(2,n)$ and let $T$ be a primitive sublattice of signature $(2,m)$ with $m<n$. Then the orthogonal complement $T_M^\perp$ is negative-definite and we have the usual inclusions 
$$
T\oplus T_M^\perp < M < M^\vee < T^\vee \oplus (T_M^\perp)^\vee.
$$
For $v\in M$ with $v^2<0$ we write 
$$
v=\alpha+\beta, \quad \alpha \in T^\vee, \; \beta \in (T_M^\perp)^\vee.
$$
Then we have
\begin{equation*}
 \cD(T)\cap \cD_v(M) = \begin{cases}
      \cD_\alpha(T),    &\text{if} \; \alpha^2<0,\\
      \emptyset,    &\text{if} \; \alpha^2\geq 0,\, \alpha\neq 0, \\
      \cD(T),    &\text{if} \; \alpha=0, \; \text{i.e.} \; v\in T_M^\perp.
      \end{cases}    
\end{equation*}

The next theorem was proved in \cite[Theorem 1.2]{BKPS98} and \cite[Theorems 8.3 and 8.18]{GHS13}.

\begin{theorem}\label{th: Borchpullback}
Let $T\hookrightarrow \II_{2,26}$ be a primitive non-degenerate sublattice of signature $(2,n)$ with $n\geq 3$, and let $\cD(T)\hookrightarrow \cD(\II_{2,26})$ be the corresponding embedding of the Hermitian symmetric domains. The set of $(-2)$-roots in the orthogonal complement
$$
R_{-2}(T^\perp)=\{ r\in \II_{2,26} : r^2=-2,\; (r,T)=0\}
$$
is finite. We put $N(T^\perp)=\frac{1}{2}|R_{-2}(T^\perp)|$. Then we have
\begin{equation}
\Phi_{12}\lvert_T=\frac{\Phi_{12}(Z)}{\prod_{r\in R_{-2}(T^\perp)/\pm 1 }(Z,r)}\Bigg\vert_{\cD(T)^\bullet} \in M_{12+N(T^\perp)}(\widetilde{\Orth}^+(T), \det),
\end{equation}
where in the product over $r$ we fix a finite system of representatives in $R_{-2}(T^\perp)/\pm 1$. The modular form $\Phi_{12}\lvert_T$ vanishes only on rational quadratic divisors of type $\cD_v(T)$ where $v\in T^\vee$ is the orthogonal projection of a $(-2)$-root $r\in \II_{2,26}$ on $T^\vee$ satisfying $-2\leq v^2 <0$. If the set $R_{-2}(T^\perp)$ is non-empty then the quasi pull-back $\Phi_{12}\lvert_T$ is a cusp form.
\end{theorem}

In general, the quasi pull-back $\Phi_{12}\lvert_T$ is not a reflective modular form. To determine its divisor, we must do explicit calculations. We refer to \cite{Gra09} for this type of calculations.  We next introduce several arguments which can be used to seek reflective modular forms without complicated calculations. 

In \cite{Gri18} Gritsenko proposed 24 Jacobi type constructions of $\Phi_{12}$ based on  $24$ one-dimensional cusps of the modular variety $\Orth^+(\II_{2,26})\backslash \cD(\II_{2,26})$.  These components correspond exactly to the classes of positive-definite even unimodular lattices of rank $24$. They are the $23$ Niemeier lattices $N(R)$ uniquely determined by their root sublattices $R$ of rank $24$
\begin{align*}
&3E_8& &E_8\oplus D_{16}& & D_{24}& &2D_{12}& &3D_8& &4D_6&\\
&6D_4& &A_{24}& &2A_{12}& &3A_8& &4A_6& &6A_4& \\
&8A_3& &12A_2& &24A_1& &E_7\oplus A_{17}& &2E_7\oplus D_{10}& &4E_6& \\
&E_6\oplus D_7\oplus A_{11}& &A_{15}\oplus D_9& &2A_9\oplus D_6& &2A_7\oplus 2D_5& &4A_5\oplus D_4&
\end{align*}
and the Leech lattice $\Lambda_{24}$ without 2-roots which we will also denote by $N(\emptyset)$ (see \cite[Chapter 18]{SC98}). We next construct a lot of reflective modular forms by quasi pull-backs of $\Phi_{12}$ at different 1-dimensional cusps,  some already known, some new.

For convenience, we fix the discriminant groups of irreducible root lattices. Let $e_1,...,e_n$ be the standard basis of $\RR^n$.
\begin{enumerate}
\item For $A_1=\ZZ$ with the bilinear form $2x^2$, we fix
$A_1^\vee/ A_1 = \{ 0, 1/2 \}.$
\item For $D_n=\{ x\in \ZZ^n: \sum_{i=1}^n x_i \in 2\ZZ \}$ with $n\geq 4$, we fix
$D_n^\vee / D_n=\{[0], [1], [2], [3] \}$,
where $[1]=\frac{1}{2}\sum_{i=1}^n e_i$, $[2]=e_1$, $[3]=\frac{1}{2} \sum_{i=1}^n e_i - e_n$. Then $[1]^2=[3]^2=\frac{n}{4}$ and $[2]^2=1$.
\item For $A_n=\{ x\in \ZZ^{n+1}: \sum_{i=1}^{n+1} x_i =0 \}$ with $n\geq 2$, we fix
$ A_n^\vee /A_n=\{[i]: 0\leq i \leq n  \}$,
where $[i]=(\frac{i}{n+1},...,\frac{i}{n+1},\frac{-j}{n+1},...,\frac{-j}{n+1} )$ with $j$ components equal to $\frac{i}{n+1}$ and $i+j=n+1$. The norm of $[i]$ is $\frac{ij}{n+1}$.
\item The lattice $E_6$ is of level $3$. Its discriminant group is of order $3$ and generated by one element $[1]$ of norm $4/3$.
\item The lattice $E_7$ is of level $4$. Its discriminant group is of order $2$ and generated by one element $[1]$ of norm $3/2$.
\end{enumerate}

We consider the special case $T=2U\oplus K(-1)$, where $K$ is a primitive sublattice of some $N(R)$. The theta-function $\Theta_{N(R)}$ is a holomorphic Jacobi form of weight $12$ and index $N(R)$. The input of $\Phi_{12}$ as a Jacobi form at the $1$-dimensional cusp related to $2U\oplus N(R)$ is defined by
$$
\varphi_{0,N(R)}(\tau,\mathfrak{z})=\frac{\Theta_{N(R)}(\tau,\mathfrak{z})}{\Delta(\tau)}=q^{-1}+24+\sum_{r\in R, r^2=2}e^{2\pi i (r,\mathfrak{z})}+O(q)\in J_{0,N(R)}^{!}.
$$
We write $\mathfrak{z}=\mathfrak{z}_1+\mathfrak{z}_2$ with $\mathfrak{z}_1\in K\otimes \CC$ and $\mathfrak{z}_2\in K_{N(R)}^\perp\otimes \CC$ and define the pull-back of $\varphi_{0,N(R)}$ on the lattice $K\hookrightarrow N(R)$ as
\begin{align*}
\varphi_{0,K}(\tau,\mathfrak{z}_1)=\varphi_{0,N(R)}(\tau,\mathfrak{z})\lvert_{\mathfrak{z}_2=0} \in J_{0,K}^{!}.
\end{align*}
By expanding $\Phi_{12}$ on the tube domain $\cH(N(R))$ and using its expression described in Theorem \ref{th:BorcherdsJF}, we see that $\Phi_{12}|_T$ equals $\Borch(\varphi_{0,K})$ up to a constant multiple (the vector $r$ in Theorem \ref{th: Borchpullback} is a $(-2)$-root of $K_{N(R)}^\perp$, $(Z,r)$ reduces to $(\mathfrak{z}_2, r)$ and it is canceled by the term $(1-\zeta^{r})$ in the product expansion of $\Phi_{12}$). This identity also follows from the general result \cite[Remark 3.5]{Ma19}.

\subsection{The first argument.}
This argument was due to Gritsenko and Nikulin.
In \cite{GN18}, they constructed modular forms with complete 2-divisor by quasi pull-backs of $\Phi_{12}$. We recall their main ideas such that readers can understand the other arguments better.

For an even positive-definite lattice $L$, we define the $\norm_2$ condition as
\begin{equation}\label{Norm2}
\norm_2:\quad \forall\; \bar{c} \in L^{\vee}/L,\quad \exists\, h_c \in \bar{c} \quad \text{s.t.} \quad 0 \leq h_c^2  \leq 2.
\end{equation}
The reason why we formulated the $\norm_2$ condition is the following. If $L$ satisfies the $\norm_2$ condition and $\phi$ is a weakly holomorphic Jacobi form of index $L$, then its singular Fourier coefficients are totally determined by the $q^n$-terms with non-positive $n$. 
\begin{proof}[Proof of the claim]
It is known that $f(n,\ell)$ depends only on the number $2n-(\ell, \ell)$
and the class $\ell \mod L$.
Suppose that $f(n,\ell)$ is singular, i.e. $2n-(\ell, \ell)<0$.
There exists a vector 
$\ell_1 \in L^\vee$ such that $(\ell_1,\ell_1)\leq 2$ 
and $\ell - \ell_1\in L$ because $L$ satisfies the $\norm_2$ condition. It is
clear that $(\ell, \ell)-(\ell_1 , \ell_1 )$ is an even integer.
If $-2\leq 2n-(\ell, \ell)< 0$, it follows that 
$2n-(\ell, \ell)=-(\ell_1,\ell_1)$ and $f(n,\ell)=f(0,\ell_1)$. 
If $2n-(\ell, \ell)<-2$, then 
there exists a negative integer $n_1$ satisfying 
$2n-(\ell, \ell)=2n_1-(\ell_1,\ell_1)$. Thus there exists a Fourier 
coefficient $f(n_1,\ell_1)$ with negative $n_1$ such that $f(n_1,\ell_1)=f(n,\ell)$.
\end{proof}

Let $K$ be a primitive sublattice of $N(R)$ containing a direct summand of $R$ which has the same rank as $K$, or let $K$ be a primitive sublattice of the Leech lattice $\Lambda_{24}=N(\emptyset)$. We assume that $K$
satisfies the $\norm_2$ condition. Recall $T=2U\oplus K(-1)$.  In this case, we have
$$
\varphi_{0,K}(\tau,\mathfrak{z}_1)= q^{-1}+24+ n_K + \sum_{r\in K,\, r^2=2}e^{2\pi i (r,\mathfrak{z}_1)}+O(q) \in J_{0,K}^{!},
$$
where $n_K$ is the number of $2$-roots in $R$ orthogonal to $K$. Since $K$ satisfies the $\norm_2$ condition, the singular Fourier coefficients of $\varphi_{0,K}$ are represented by its $q^{-1}$ and $q^0$-terms. We then see that $\Borch(\varphi_{0,K})$ vanishes precisely on the divisor $\cH$ defined in \eqref{Heegner}. Therefore, $\Phi_{12}|_T$ is a modular form with complete $2$-divisor. In this way, Gritsenko and Nikulin \cite[Theorems 4.3, 4.4]{GN18}  constructed the following modular forms with complete $2$-divisor.

\begin{table}[ht]
\caption{Reflective cusp forms with complete 2-divisor}
\label{tab: GN1}
\renewcommand\arraystretch{1.5}
\noindent\[
\begin{array}{|c|c|c|c|c|c|c|c|c|c|c|c|c|c|c|}
\hline 
\text{lattice}& A_1& 2A_1& 3A_1& 4A_1& N_8& A_2& 2A_2& 3A_2& A_3& 2A_3& A_4& A_5\\
\text{weight}& 35& 34& 33& 32& 28& 45& 42& 39& 54& 48& 62& 69\\
\hline
\text{lattice}& A_6& A_7& D_4& 2D_4&D_5& D_6&D_7&D_8&E_6&E_7&E_8&2E_8\\
\text{weight}& 75& 80& 72& 60& 88&102&114&124&120&165&252&132\\
\hline
\end{array} 
\]
\end{table}
Here, $N_8$ is the Nikulin lattice defined as (see \cite[Example 4.3]{GN18})
\begin{equation}\label{eq:Nikulin lattice}
N_8=\latt{8A_1, h=(a_1+\cdots + a_8)/2}\cong D_8^\vee (2),
\end{equation}
where $(a_i,a_j)=2\delta_{ij}$, $(h,h)=4$. The root sublattice generated by the 2-roots of $N_8$ is $8A_1$.

When $K$ is one of the following $10$ sublattices of the Leech lattice $\Lambda_{24}$
\begin{align*}
&A_1(2)& &A_1(3)& &A_1(4)& &2A_1(2)& &A_2(2)& 
&A_2(3)& &A_3(2)& &D_4(2)& &E_8(2)& &A_4^\vee(5)&
\end{align*}
there exists a (non-cusp) modular form of weight $12$ with complete 2-divisor for $2U\oplus K(-1)$. The fact that  $A_4^\vee(5)$ satisfies the $\norm_2$ condition was proved in \cite{GW17, GW18}.

\subsection{The second argument.}

This argument was formulated in \cite{Gri18}. We here describe it in a more understandable way and use it to construct much more 2-reflective and reflective modular forms. This argument is based on the following observation.

\bigskip
\noindent
\textbf{Observation.} Let $L=A_1, 2A_1, A_2, D_4, A_1(2)$. For any $0\neq \gamma \in L^\vee/L$, the minimal norm vector $v$ in $\gamma +L$, i.e. $v$ satisfying
$$
(v,v)\leq (u,u), \quad \text{for all $u\in \gamma +L$},
$$
is reflective, namely $\sigma_v \in \Orth(L)$. When $L=A_1$, $\sigma_v\in \widetilde{\Orth}(L)=\Orth(L)$. 
\bigskip

Let $K=K_0\oplus K_1 \oplus K_2$ be a primitive sublattice of $N(R)$. We assume:
\begin{itemize}
    \item[(i)] The lattice $K_0$ contains a direct summand of $R$ which has the same rank as $K_0$;
    \item[(ii)] The lattices $K_1$, $K_2$ take $A_1$, $2A_1$, $A_2$, $D_4$ or $A_1(2)$, and they are contained in different direct summands of $R$. The second lattice $K_2$ is allowed to be empty;
    \item[(iii)] The lattice $K$ satisfies the $\norm_2$ condition. 
\end{itemize}

Recall $T=2U\oplus K(-1)$. Again, we consider the pull-back of $\varphi_{0,N(R)}$ on $K\hookrightarrow N(R)$. The above assumptions guarantee that the singular Fourier coefficients of $\varphi_{0,K}$ are totally determined by its $q^{-1}$, $q^0$-terms and correspond to reflective divisors.  Therefore, the quasi pull-back $\Phi_{12}\lvert_T$, i.e. $\Borch(\varphi_{0,K})$ is a reflective modular form. 

We first use this argument to construct $2$-reflective modular forms. To do this, we can only take $K_1, K_2= \emptyset$ or $A_1$.  Let $R=3E_8$, $K_0=2E_8$ and $K_1=A_1$ contained in the third copy of $E_8$. Then the quasi pull-back $\Phi_{12}\lvert_T$ gives a $2$-reflective modular form for $2U\oplus 2E_8(-1)\oplus A_1(-1)$. Similarly, when $K$ equals one of the following $16$ lattices
\begin{align*}
&2E_8\oplus A_1& &E_8\oplus A_1& &E_8\oplus 2A_1& &D_6\oplus A_1& &D_4\oplus A_1& &D_4\oplus 2A_1& &A_4\oplus A_1& &A_3\oplus A_1&\\
&A_3\oplus 2A_1&  &A_2\oplus A_1& &A_2\oplus 2A_1& &2A_2\oplus A_1& &D_5\oplus A_1& &A_5\oplus A_1& &E_7\oplus A_1& &E_6\oplus A_1&
\end{align*}
the quasi pull-back $\Phi_{12}\lvert_T$ gives a $2$-reflective modular form on $T$.

We then use this argument to construct reflective modular forms.  For instance, let $R=3E_8$, $K_0=E_8$ and $K_1=K_2=D_4$ contained in the second and third copies of $E_8$ respectively. Then $\Phi_{12}\lvert_T$ gives a reflective modular form for $2U\oplus E_8(-1)\oplus 2D_4(-1)$. When $K$ is equal to one of the following $33$ lattices, $\Phi_{12}\lvert_T$ is a reflective modular form on $T$.
\begin{enumerate}
\item When $R=3E_8$, the lattice $K$ is equal to
\begin{align*}
&\{E_8, 2E_8\}\oplus \{2A_1, A_2, D_4, A_1(2)\}& &E_8\oplus A_1 \oplus \{2A_1, A_2, D_4, A_1(2) \}&\\
&E_8\oplus 2A_1 \oplus \{2A_1, A_2, D_4, A_1(2) \}& &E_8\oplus A_2 \oplus \{A_2, D_4, A_1(2) \}& \\
&E_8\oplus D_4 \oplus \{D_4, A_1(2) \}& &E_8\oplus 2A_1(2).&
\end{align*}
\item  When $R=6D_4$, the lattice $K$ is equal to
$D_4\oplus \{ A_2, D_4, A_1(2) \}.$
\item  When $R=6A_4$, the lattice $K$ is equal to $A_4\oplus A_2$. 
\item When $R=8A_3$, the lattice $K$ is equal to
$A_3\oplus \{ A_2, A_1(2) \}.$
\item When $R=12A_2$, the lattice $K$ is equal to $A_2\oplus A_1(2)$.
\item When $R=24A_1$, the lattice $K$ is equal to $\{A_1, 2A_1\}\oplus A_1(2)$.
\item When $R=4E_6$, the lattice $K$ is equal to $E_6\oplus A_2$.
\item When $R=2A_7\oplus 2D_5$, the lattice $K$ is equal to $D_5\oplus A_2$.
\end{enumerate}

For (5) and (6), the constructions are a bit different. We take $A_1(2)$ in (5) as a sublattice of $2A_2$ and take $A_1(2)$ in (6) as a sublattice of $2A_1$. They can also be constructed in another way. For example, to construct a reflective modular form for $2U\oplus A_2\oplus A_1(2)$, we can use the pull-back $A_2\oplus A_1(2)< N(6D_4)$.
Remark that for some lattice we may construct more than one reflective modular forms. In the above, we focus on reflective lattices and only construct one reflective modular form for a certain lattice.

\subsection{The third argument.}

We now consider the general case, i.e. the lattice $K$ does not satisfy the $\norm_2$ condition. Assume that the lattice $K$ satisfies conditions (i) and (ii) of the second argument. We further assume that the minimum norm of vectors in any nontrivial class of the discriminant group of $K$ is less than $4$ and all vectors (noted by $v$) of minimum norm larger than $2$ satisfy the condition: \textit{the vector $(0,1,v,1,0)$ is reflective.}
In this case, the singular Fourier coefficients of $\varphi_{0,K}$ are determined by its $q^{-1}$, $q^0$, $q^1$-terms and correspond to reflective divisors. Therefore, the quasi pull-back $\Phi_{12}\lvert_T$, i.e. $\Borch(\varphi_{0,K})$ is a reflective modular form. 
\begin{enumerate}
\item When the lattice $K$ is equal to one of the following $8$ lattices, we get $2$-reflective lattices.
\begin{align*}
&5A_1& &D_{10}& &N_8\oplus A_1&  &D_8\oplus A_1& \\
&D_6\oplus 2A_1& 
&2D_4\oplus A_1&  &E_7\oplus 2A_1&  &D_4\oplus 3A_1& 
\end{align*}
For the last lattice, we use $R=6D_4$ and take $3A_1$ from three different copies of $D_4$.

\item  When the lattice $K$ is equal to one of the following $12$ lattices, we get reflective lattices.
$$
4A_2 \quad 3D_4 \quad 2E_6 \quad 2E_7 \quad A_8 \quad A_9 \quad D_9 \quad A_5\oplus D_4 \quad D_{12} \quad 2A_4 \quad 2D_5 \quad 2D_6
$$
There are a lot of this type of reflective lattices. In the above, we only consider the simplest case $K=K_0$. By \cite[\S 9]{Sch04}, the lattice $2A_2(2)$ is a primitive sublattice of the Leech lattice and it satisfies our condition. Thus the quasi pull-back gives a reflective modular form of weight $12$ for $2U\oplus 2A_2(-2)$.
\end{enumerate}

\subsection{The fourth argument.}\label{subsec:4th argument}

We can also consider the quasi pull-backs of some other reflective modular forms. We have known that $2U\oplus 2E_8\oplus D_4$ is reflective. It is easy to check  that
$$
2U\oplus 2E_8\oplus D_4 \cong 2U\oplus E_8 \oplus D_{12}
$$
because they have the same discriminant form. For $4\leq n \leq 10$, we have $D_n\oplus D_{12-n} < D_{12}$. In a similar way, we show that the quasi pull-back of $2U\oplus E_8 \oplus D_{12}$ into $2U\oplus E_8\oplus D_n$ will give a reflective modular form on $2U\oplus E_8 \oplus D_n$ with $4\leq n \leq 10$. 

\subsection{The fifth argument.}

This argument relies on the construction of the Niemeier lattice $N(R)$ from some root lattice $R$. We will explain the main idea by considering several interesting examples. 

\smallskip
\textbf{(1)} Let $R=6D_4$. We consider its sublattice $K=D_4\oplus 5A_1$, where every $A_1$ is contained in a different copy of $D_4$. The singular Fourier coefficients of $\varphi_{0,K}$ are determined by its $q^{-1}$, $q^0$, $q^1$-terms. It is clear that the $q^{-1}$ and $q^0$-terms correspond to 2-reflective divisors. We next consider the $q^1$-term, which is the pull-back of vectors of norm $4$ in $N(6D_4)$. Since the pull-backs of vectors of norm $4$ in $6D_4$ gives either non-singular Fourier coefficients or singular Fourier coefficients equivalent to that of the $q^0$-term, we only need to consider the pull-backs of vectors of norm $4$ in $N(6D_4)$ and not in $6D_4$.  This type of vectors has the form $[i_1]\oplus[i_2]\oplus [i_3] \oplus [i_4]\oplus [i_5]\oplus [i_6]$, here four of the six indices are non-zero.  Its pull-back to $D_4\oplus 5A_1$ only gives the singular Fourier coefficients of type $[i]\oplus (\frac{1}{2},\frac{1}{2}, \frac{1}{2}, 0, 0 )$, which correspond to $2$-reflective divisors. Therefore, $\Borch(\varphi_{0,K})$ is a 2-reflective modular form for $2U\oplus D_4\oplus 5A_1$. 

Similarly, the quasi pull-back on $6A_2<N(6D_4)$ gives a reflective modular form for $2U\oplus 6A_2$.

\smallskip
\textbf{(2)} Let $R=8A_3$. We consider its sublattice $K=8A_1$, where every $A_1$ is contained in a different copy of $A_3$. The singular Fourier coefficients of $\varphi_{0,K}$ are determined by its $q^{-1}$, $q^0$, $q^1$-terms. It is obvious that the $q^{-1}$ and $q^0$-terms correspond to 2-reflective divisors. We next consider the $q^1$-term. The $q^1$-term is the pull-back of vectors of norm $4$ in $N(8A_3)$. Due to the similar reason, we only need to consider the pull-back of vectors of norm $4$ in $N(8A_3)$ and not in $8A_3$. This type of vectors has the form $[2]\oplus[2]\oplus [2] \oplus [2]\oplus 0^4$ or $[2]\oplus[1]\oplus [1] \oplus [1]\oplus [1]\oplus 0^3$.  Their pull-backs to $8A_1$ only give the singular Fourier coefficients of type $(\frac{1}{2},\frac{1}{2}, \frac{1}{2},\frac{1}{2},\frac{1}{2}, 0, 0,0 )$, which correspond to $2$-reflective divisors. Therefore,  $\Borch(\varphi_{0,K})$ is a 2-reflective modular form for $2U\oplus 8A_1$.

\smallskip
\textbf{(3)} Let $R=12A_2$. We consider its sublattice $K=12A_1$, where every $A_1$ is contained in a different copy of $A_2$. The singular Fourier coefficients of $\varphi_{0,K}$ are determined by its $q^{-1}$, $q^0$, $q^1$, $q^2$-terms. Firstly, the $q^{-1}$ and $q^0$-terms correspond to 2-reflective divisors. We next consider the $q^1$ and $q^2$-terms. To find singular Fourier coefficients in the $q^1$-term, we only need to consider the pull-back of vectors of norm $4$ in $N(12A_2)$ and not in $12A_2$. This type of vectors has the form $[1]^6 \oplus 0^6$.  Its pull-back to $12A_1$ only gives the singular Fourier coefficients of type $(\frac{1}{2},\frac{1}{2}, \frac{1}{2},\frac{1}{2},\frac{1}{2}, 0^7)$ or $(\frac{1}{2},\frac{1}{2}, \frac{1}{2},\frac{1}{2},\frac{1}{2},\frac{1}{2}, 0^6)$, which all correspond to reflective divisors. The $q^2$-term is the pull-back of vectors of norm $6$ in $N(12A_2)$. This type of vectors has the form $[1]^9 \oplus 0^3$.  Its pull-back to $12A_1$ only gives the singular Fourier coefficients of type $(\frac{1}{2},\frac{1}{2}, \frac{1}{2},\frac{1}{2},\frac{1}{2}, \frac{1}{2},\frac{1}{2}, \frac{1}{2},\frac{1}{2}, 0^3)$, which correspond to 2-reflective divisors.
Therefore, $\Borch(\varphi_{0,K})$ is a reflective modular form for $2U\oplus 12A_1$.

\smallskip
Using the idea of pull-backs, it is easy to prove the following result.
\begin{lemma}\label{Lem:pullbacktest}
Let $M$ be an even lattice of signature of $(2,n)$ and $L$ be an even negative-definite lattice. If $M\oplus L$ is reflective (resp. $2$-reflective), then $M$ is reflective (resp. $2$-reflective) too.
\end{lemma}

As an application of this lemma, we construct more reflective modular forms. We check that
$$
U\oplus U(2)\oplus D_4 < U\oplus U(2)\oplus D_4\oplus D_4 \cong 2U\oplus N_8,
$$
which yields that $U\oplus U(2)\oplus D_4$ is 2-reflective. Similarly, we claim that $U\oplus U(3)\oplus A_2$ is reflective because
$$
U\oplus U(3)\oplus A_2 < U\oplus U(3)\oplus A_2\oplus E_6 \cong 2U\oplus 4A_2.
$$

For any reflective modular form constructed in this section, it is possible to work out the weight and the multiplicities of zero divisors by the methods used in the next section.

\section{Classification of 2-reflective modular forms}\label{Sec:classification2}
In this section we use the approach based on Jacobi forms to classify $2$-reflective modular forms on lattices containing two hyperbolic planes. 

\subsection{Known results}
We first review some results proved in \cite{Wan18}. Let $M=2U\oplus L(-1)$ be an even lattice of signature $(2,\rank(L)+2)$. Let $\pi_M\subset D(M)$ be the subset of elements of order $2$ and norm $-1/2$. For each $\mu \in \pi_M$ we abbreviate $\cH(\mu,-1/4)$ by $\cH_\mu$. We also set 
\begin{equation}\label{eq:H0}
\cH_0 = \bigcup_{\substack{\ell \in M, (\ell,\ell)=-2\\ \div (\ell)=1}} \ell^{\perp}\cap \cD(M). 
\end{equation}
Then we have the following decomposition
\begin{equation}
\cH= \cH_0+ \sum_{\mu \in \pi_M} \cH_\mu.
\end{equation}
Assume that $F$ is a $2$-reflective modular form of weight $k$ for $M$. Then its divisor can be written as 
\begin{equation}\label{eq:divisor}
\div (F)= \beta_0 \cH_0 + \sum_{\mu \in \pi_M} \beta_\mu \cH_\mu 
         =\beta_0 \cH + \sum_{\mu \in \pi_M} (\beta_\mu-\beta_0) \cH_\mu,
\end{equation}
where $\beta_*$ are non-negative integers.
By \cite[Theorem 5.12]{Bru02} or \cite[Theorem 1.2]{Bru14}, the modular form $F$ should be a Borcherds product. In view of the isomorphism between vector-valued modular forms and Jacobi forms, there exists a weakly holomorphic Jacobi form $\phi_L$ of weight $0$ and index $L$ with singular Fourier coefficients of the form 
\begin{equation}\label{F2}
\sing(\phi_L)= \beta_0 \sum_{r\in L}q^{(r,r)/2-1}\zeta^r+ \sum_{\mu \in \pi_M}(\beta_\mu - \beta_0)\sum_{s\in L+ \mu}q^{(s,s)/2-1/4}\zeta^s,
\end{equation}
where $\zeta^l=e^{2\pi i (l,\mathfrak{z})}$.
We thus obtain
\begin{equation}\label{F3}
\phi_L(\tau,\mathfrak{z})= \beta_0 q^{-1}+\beta_0 \sum_{r\in R_L}\zeta^r +2k+ \sum_{\substack{u\in \pi_M}} (\beta_\mu -\beta_0) \sum_{s\in R_\mu (L)}\zeta^s + O(q),
\end{equation}
here $R_L$ is the set of 2-roots of $L$ and 
\begin{equation}\label{F4}
R_\mu (L)=\{ s\in L^\vee : 2s\in R_L, s\in L+\mu \}.
\end{equation}

In \cite[Theorem 3.2]{Wan18}, we proved that the weight of $F$ is given by
\begin{equation}\label{eq:weight}
k=\beta_0\left[ 12+ \lvert R_L \rvert \left(\frac{12}{\rank(L)}-\frac{1}{2} \right) \right]
+\left(\frac{3}{\rank(L)}-\frac{1}{2} \right)\sum_{\mu \in \pi_M} (\beta_\mu - \beta_0 )\lvert R_\mu (L) \rvert.
\end{equation}

\subsection{Nonexistence of 2-reflective lattice of signature (2,14)} 
We refine the proof of \cite[Theorem 3.6]{Wan18} to demonstrate the following result.

\begin{theorem}\label{th:nonsign12}
There is no $2$-reflective lattice of signature $(2,14)$.
\end{theorem}

\begin{proof}
Suppose that $M$ is a $2$-reflective lattice of signature $(2,14)$. Without loss of generality, we can assume that $M$ is maximal, namely $M$ has no any proper even overlattice. As the proof of \cite[Proposition 3.1]{Ma17}, we can show that $M$ contains two hyperbolic planes. It means that $M$ can be written as $M=2U\oplus L(-1)$. Therefore, there exists a weakly holomorphic Jacobi form of weight $0$ and index $L$, noted by $\phi$. Like the proof of \cite[Theorem 3.6]{Wan18}, applying differential operators defined in Lemma \ref{Lem:differential} to $\phi$, we construct weakly holomorphic Jacobi forms of weights $2$ and $6$ as $\phi_2:=H_0(\phi)$ and $\phi_6:=H_4H_2H_0(\phi)$. We see from its construction that $H_k$ preserves the singular Fourier coefficients of $\phi$. There are two types of singular Fourier coefficients (with hyperbolic norms $-2$ and $-1/2$) in these $\phi_i$.  Let $E_4$ and $E_6$ denote the Eisenstein series of weights $4$ and $6$ on $\SL_2(\ZZ)$. By taking a suitable $\CC$-linear combination of $E_6\phi_0$, $E_4\phi_2$ and $\phi_6$ to cancel the two types of singular Fourier coefficients, we construct a holomorphic Jacobi form of weight $6$ and index $L$, denoted by $\varphi_6$. By direct calculation, the constant term of $\varphi_6$ is not zero and we assume it to be $1$. The function $\varphi_6$ has singular weight $6$.  Thus, it is a $\CC$-linear combination of theta-functions for $L$ defined as \eqref{eq:ThetaFunction}. Since $L$ is maximal, there is no $\gamma\in L^\vee$ such that $\gamma\not\in L$ and $(\gamma,\gamma)=2$. Thus the $q^1$-term in the Fourier expansion of $\varphi_6$ comes only from the theta-function $\Theta_0^L$. But $\varphi_6(\tau,0)=E_6(\tau)=1-504q+...$, which leads to a contradiction. We complete the proof.
\end{proof}

\subsection{More refined results}
In this subsection, we prove the following main result. 

\begin{theorem}\label{th:2-reflective}
Let $M=2U\oplus L(-1)$ and $F$ be a $2$-reflective modular form of weight $k$ with divisor of the form \eqref{eq:divisor} for $M$. Let $R(L)$ be the root sublattice generated by the $2$-roots of $L$. If $R(L)$ is empty, then $k=12\beta_0$. If $R(L)$ is non-empty, then $R(L)$ and $L$ have the same rank, which is denoted by $n$. Furthermore, the root lattice $R(L)$ satisfies one of the following conditions
\begin{enumerate}
\item[(a)] $R(L)=nA_1$. In this case, all $\beta_\mu$ satisfying $R_\mu(L)\neq \emptyset$ are the same. 

\item[(b)]  The lattice $A_1$ is not an irreducible component of $R(L)$. In this case, all the irreducible components of $R(L)$ have the same Coxeter number, which is denoted by $h$. In addition, the sets $R_\mu(L)$ are all empty and the weight $k$ is given by
\begin{equation*}
k=\beta_0\left(12+12h-\frac{1}{2}nh\right).
\end{equation*}

\item[(c)] $R(L)=mA_1\oplus R$, where $1\leq m \leq n-2$ and the lattice $A_1$ is not an irreducible component of $R$. In this case, all the irreducible components of $R$ have the same Coxeter number, which is denoted by $h$. In addition, all $\beta_\mu$ satisfying $R_\mu(L)\neq \emptyset$ are the same, which is denoted by $\beta_1$. Moreover, we have
\begin{align*}
\beta_1&=(2h-3)\beta_0,\\
k&=\beta_0\left[ \left( 12 - \frac{n+3m}{2} \right)h +12+3m  \right].
\end{align*}
Furthermore, $L$ can be represented as $mA_1\oplus L_0$, where $L_0$ is an even overlattice of $R$.
\end{enumerate}
\end{theorem}

\begin{proof}
First, if $R(L)=\emptyset$ then we derive from \eqref{eq:weight} that $k=12\beta_0$.

We next assume that $R(L)\neq \emptyset$. Let $R(L)=\oplus R_i$ be the decomposition of irreducible components i.e. $R_i$ are irreducible root lattices.  We write $\mathfrak{z}=\sum_i \mathfrak{z}_i \in L\otimes \CC$, where $\mathfrak{z}_i\in R_i\otimes\CC$. For irreducible root lattices,  only the lattice $A_1$ satisfies the property that there is a root $v$ such that $v/2$ is in the dual lattice.
By \eqref{eq:vectorsystem} and \eqref{F3}, we conclude that $R(L)$ and $L$ have the same rank. Otherwise, there exists a vector in $L\otimes\CC$ orthogonal to $R(L)\otimes\CC$, which contradicts \eqref{eq:vectorsystem} because the number $C$ is not equal to $0$. In a similar way, we can prove statement (a).

We next prove statement (b). Since there is no $R_i$ equal to $A_1$,  the sets $R_\mu(L)$ are all empty. By Lemma \ref{Lem:q^0-term} and \eqref{F3}, we have 
\begin{align*}
\sum_i \sum_{\substack{r\in R_i\\ r^2=2}} \beta_0 (r, \mathfrak{z}_i)^2=\frac{2}{n}\sum_i \beta_0 h_i \rank(R_i)\sum_i (\mathfrak{z}_i, \mathfrak{z}_i),
\end{align*}
where $h_i$ are the Coxeter numbers of $R_i$. On the other hand, by Proposition \ref{prop:root}, we have
$$
\sum_{\substack{r\in R_i\\ r^2=2}}(r, \mathfrak{z}_i)^2= 2h_i (\mathfrak{z}_i, \mathfrak{z}_i).
$$ 
Thus, all Coxeter numbers $h_i$ are the same. The weight formula follows from \eqref{eq:weight}. 

We now prove statement (c). Firstly, all non-empty $R_\mu(L)$ are contained in the components $mA_1$. We write $R=\oplus R_j$, where $R_j$ are irreducible root lattices. For $1\leq t \leq m$, if the dual lattice of the $t$-th copy of $A_1$ is contained in $L^\vee$, then the corresponding $R_\mu(L)$ is non-empty and have two elements. In this case, we denote the associated $\beta_\mu$  by $\beta_t$. If not, then the corresponding $R_\mu(L)$ is empty, so we have $\beta_t:=\beta_\mu=\beta_0$. We also denote the elliptic parameter associated to the $t$-th copy of $A_1$ by $\mathfrak{z}_t$. By Lemma \ref{Lem:q^0-term} and \eqref{F3}, we have 
\begin{align*}
\sum_j \sum_{\substack{r\in R_j\\ r^2=2}} \beta_0 (r, \mathfrak{z}_j)^2 +\sum_t 2( (\beta_t-\beta_0) + 4\beta_0) \mathfrak{z}_t^2 =2C\left[ \sum_j (\mathfrak{z}_j, \mathfrak{z}_j)+ \sum_t 2\mathfrak{z}_t^2\right].
\end{align*}
In the above identity, we use the standard model of $A_1$: $A_1=\ZZ\alpha$ with $\alpha^2=2$. Let $h_j$ denote the Coxeter number of $R_j$. Then we have
$$
C=\beta_0 h_j=\frac{1}{2}\beta_t+\frac{3}{2}\beta_0, \quad \forall\ h_j, \ \forall \ 1\leq t \leq m.
$$
Therefore, all $h_j$ are the same (noted by $h$), and all $\beta_t$ are also the same (noted by $\beta_1$), which implies that the dual lattice of each copy of $A_1$ is contained in $L^\vee$. It follows that $\beta_1=(2h-3)\beta_0$. Combining the formula $\beta_1=(2h-3)\beta_0$ and \eqref{eq:weight} together, we deduce the weight formula. We set $L_0=\{ v\in L: (v,x)=0, \forall x\in mA_1 \}$. Then we have
$$
mA_1\oplus L_0 < L < L^\vee < mA_1^\vee\oplus L_0^\vee.
$$
For any $l\in L$, we can write $l=l_1+l_2$ with $l_1\in mA_1^\vee$ and $l_2\in L_0^\vee$. Since $mA_1^\vee < L^\vee$, we have $(l, mA_1^\vee) \in \ZZ$. Thus $(l_1, mA_1^\vee)\in \ZZ$, which yields $l_1\in mA_1$. Therefore, $l_2=l-l_1\in L$ and then $l_2\in L_0$ due to $(l_2, mA_1)=0$. We thus prove $L=mA_1\oplus L_0$.
\end{proof}

Inspired by the classification of even positive-definite unimodular lattices, we define the following classes of $2$-reflective lattices.
\begin{definition}
A $2$-reflective lattice $M=2U\oplus L(-1)$ is called Leech type if $R(L)=\emptyset$. The lattice $M$ is called Niemeier type if it satisfies the condition in statement (b) and called quasi-Niemeier type if it satisfies the condition in statement (c). 
\end{definition}

In a similar way, we can prove the following necessary condition for a lattice to be reflective. This condition would be useful to classify reflective lattices.
\begin{proposition}\label{prop:reflectivetype}
If the lattice $2U\oplus L(-1)$ is reflective, then either $\mathfrak{R}_L$ is empty, or $\mathfrak{R}_L$ generates the space $L\otimes \RR$, where $\mathfrak{R}_L$ is the root system of $L$
$$
\mathfrak{R}_L=\{r \in L:  \text{$r$ is primitive,} \; \sigma_r \in \Orth(L) \}.
$$
\end{proposition}

Remark that the above result is an analogue of a result in \cite{Vin72} which stated that if $U\oplus L(-1)$ is hyperbolic reflective then $\mathfrak{R}_L$ always generates $L\otimes \RR$.

In \cite[Theorem 3.6]{Wan18}, we have shown that if $2U\oplus L(-1)$ is a $2$-reflective lattice of signature $(2,19)$ then the weight of the corresponding $2$-reflective modular form is $75\beta_0$. We next use the above theorem to prove the following refined classification, which answers a question formulated in \cite[Questions 4.13 (3)]{Wan18}.

\begin{theorem}\label{th:sign19}
If $M$ is a $2$-reflective lattice of signature $(2,19)$, then it is isomorphic to the lattice $2U\oplus 2E_8(-1)\oplus A_1(-1)$.
\end{theorem}

\begin{proof}
We first prove the assertion under the assumption that $M$ contains $2U$, i.e. $M=2U\oplus L(-1)$. It is clear that $R(L)$ is non-empty because the weight is $75\beta_0$. In addition, we have $R(L)\neq 17A_1$, otherwise the weight of the associated $2$-reflective modular form is
$$
k\leq \beta_0\left(12+ 2\cdot 17 \left( \frac{12}{17}-\frac{1}{2} \right) \right)-2\cdot 17 \beta_0  \left( \frac{3}{17}-\frac{1}{2} \right)=30\beta_0.
$$

If $M$ is of quasi-Niemeier type, then we have
$$
k=\beta_0\left[ \left( 12 - \frac{17+3m}{2} \right)h +12+3m  \right]=75\beta_0,
$$
from which it follows that $h(7-3m)/2+3m=63$. Since $1\leq m \leq 15$, the only solution is $m=1$ and $h=30$. By Table \ref{tab:Coxeter}, $R(L)=2E_8\oplus A_1$ or $D_{16}\oplus A_1$. But the lattices $D_{16}$ and $E_8\oplus D_8$ are in the same genus. Thus, $2U\oplus D_{16}\oplus A_1 \cong 2U\oplus E_8\oplus D_8\oplus A_1$. If $L=D_{16}\oplus A_1$, then the lattice $2U\oplus E_8\oplus D_8\oplus A_1$ is $2$-reflective, which contradicts Theorem \ref{th:2-reflective} (c) because $E_8$ and $D_8$ have distinct Coxeter numbers. The only non-trivial even overlattice of $D_{16}$ is the unimodular lattice $D_{16}^+$. Since $2U\oplus D_{16}^+ \oplus A_1 \cong 2U\oplus 2E_8\oplus A_1$, we prove the theorem in this case.

If $M$ is of Niemeier type, then we have 
$$
k=\beta_0\left(12+12h-17h/2 \right)=75\beta_0,
$$
which implies $h=18$. By Table \ref{tab:Coxeter}, we have $R(L)=A_{17}$ or $R(L)=D_{10}\oplus E_7$. If $L=D_{10}\oplus E_7$, then $2U\oplus E_8\oplus E_7\oplus 2A_1$ is $2$-reflective, as $D_{10}$ and $E_8\oplus 2A_1$ are in the same genus. This leads to a contradiction by Theorem \ref{th:2-reflective} (c).  If $L=A_{17}$, the $2$-reflective vector $v$ with $\div(v)=2$ is represented as $(0,2,[9], 1, 0)$ which appears in the $q^2$-term of the corresponding Jacobi form $\phi_{0,A_{17}}$ of weight $0$. In addition, the $q^1$-term of $\phi_{0,A_{17}}$ has no singular Fourier coefficients of hyperbolic norm $-1/2$. On the other hand, the Niemeier lattice $N(A_{17}\oplus E_7)$ is generated over $A_{17}\oplus E_7$ by the isotropic subgroup
$$
G=\{ [0]\oplus [0], [3]\oplus [1], [6]\oplus [0], [9]\oplus [1],  [12]\oplus [0], [15]\oplus [1]  \}.
$$
Thus, the pull-back of $\varphi_{0,N(A_{17}\oplus E_7)}$ on $A_{17}$ will give a weakly holomorphic Jacobi form $\psi_{0,A_{17}}$ of weight $0$ which has the same $q^{-1}$ and $q^0$-terms as $\phi_{0,A_{17}}$ and its singular Fourier coefficients in the $q^1$-term are represented by $[3]$ and $[6]$. The reason why these two Jacobi forms have the same $q^{-1}$ and $q^0$-terms is that they have the same type of $2$-reflective divisors and the corresponding coefficients are determined by the formulas in Lemma \ref{Lem:q^0-term}. We will often use this argument later. Thus, $\phi:=(\phi_{0,A_{17}}-\psi_{0,A_{17}})/\Delta$ is a weak Jacobi form of weight $-12$ and index $1$ for $A_{17}$. We can assume that it is invariant under the orthogonal group $\Orth(A_{17})$ by considering its symmetrization. The $q^0$-term of $\phi$  contains five $\Orth(A_{17})$-orbits: $[0]$, $[1]$,  $[2]$, $[3]$, $[6]$. By \cite[Theorem 3.6]{Wir92} or \cite[Theorem 3.1]{Wan21}, the space of weak Jacobi forms of index $1$ for $A_{17}$ invariant under $\Orth(A_{17})$ is a free module generated by ten Jacobi forms of weights $0$, $-2$, $-4$, ..., $-18$ over the ring of $\SL_2(\ZZ)$ modular forms (Jacobi forms of odd weight and index $A_{17}$ are anti-invariant under $\Orth(A_{17})$.). The ten generators were constructed in \cite{Ber00}. Note that there are ten independent $\Orth(A_{17})$-orbits appearing in $q^0$-terms of these generators, namely $[i]$ for $0\leq i \leq 9$. Therefore, the $q^0$-terms of the ten generators are linearly independent over $\CC$. There are only three independent weak Jacobi forms of weight $-12$ and index $A_{17}$. We work out their $q^0$-terms and find that the $q^0$-term of any weak Jacobi form of weight $-12$ contains at least eight orbits. But the $q^0$-term of $\phi$ has only five orbits, which leads to a contradiction. 

We complete the proof of the particular case by the fact that if $L_1$ is a nontrivial even overlattice of $R(L)=A_{17}$ or $D_{10}\oplus E_7$ then $2U\oplus L_1$ is of determinant $2$ and isomorphic to $2U\oplus 2E_8\oplus A_1$. 

We now consider the remaining case that $M$ does not contain $2U$. By \cite[Lemma 1.7]{Ma18}, there exists an even overlattice $M'$ of $M$ containing $2U$. By Lemma \ref{Lem:reductionMa}, $M'$ is also 2-reflective and then it is isomorphic to $2U\oplus 2E_8(-1)\oplus A_1(-1)$. We claim that the order of the group $M'/M$ is not a prime, otherwise $|D(M)|$ would be $2p^2$ and $M$ would contain $2U$ by Lemma \ref{lem:2U}. Thus, there exists an even lattice $M_1$ such that $M\subset M_1 \subset M'$ and $M'/M_1$ is a nontrivial cyclic group. Then $M_1$ contains $2U$ by Lemma \ref{lem:2U}. It follows that $M_1$ is 2-reflective but not isomorphic to $2U\oplus 2E_8(-1)\oplus A_1(-1)$, which contradicts the previous case. This completes the proof. 
\end{proof}

\subsection{Classification of 2-reflective lattices of Niemeier type}
In this subsection we classify $2$-reflective lattices of Niemeier type. We first consider the case of $L=R(L)$ and then consider their overlattices. We discuss case by case. Let $M=2U\oplus L(-1)$.
By \cite[Theorem 3.6]{Wan18}, Theorems \ref{th:nonsign12} and \ref{th:sign19}, if $M$ is $2$-reflective, then either $M$ is isomorphic to one of the three lattices: $\II_{2,18}$, $\II_{2,26}$ and $2U\oplus 2E_8(-1)\oplus A_1(-1)$, or we have $\rank(L)\leq 11$. Therefore, we only need to consider the case of $\rank(L)\leq 11$.

\smallskip
\textbf{(1) h=3:} The unique irreducible root lattice of Coxeter number $3$ is $A_2$. By \S \ref{sec:quasipullback},  $M$ is $2$-reflective if $L=A_2$, $2A_2$, $3A_2$. The lattice $M$ is not $2$-reflective for $L=mA_2$ with $m\geq 4$. Otherwise, since $4A_2<E_6\oplus A_2$, $2U\oplus E_6\oplus A_2$ would also be $2$-reflective, which is impossible because $E_6$ and $A_2$ have different Coxeter numbers. Then we prove this claim by Lemma \ref{Lem:pullbacktest}.

\smallskip
\textbf{(2) h=4:} The unique irreducible root lattice of Coxeter number $4$ is $A_3$. By \S \ref{sec:quasipullback}, $M$ is $2$-reflective if $L=A_3$, $2A_3$. The lattice $M$ is not $2$-reflective for $L=3A_3$ because we observe from their extended Coxeter-Dynkin diagrams that $3A_3< D_6\oplus A_3$.

\smallskip
\textbf{(3) h=5:} The unique irreducible root lattice of Coxeter number $5$ is $A_4$. By \S \ref{sec:quasipullback}, $M$ is $2$-reflective if $L=A_4$. We claim that $M$ is not $2$-reflective for $L=2A_4$. Otherwise, the lattice $2U(5)\oplus 2A_4^\vee(5)\cong U\oplus U(5) \oplus E_8(5)$ would have a reflective modular form with $10$-reflective divisors because $(2U\oplus 2A_4)^\vee(5)= 2U(5)\oplus 2A_4^\vee(5)$ and the $2$-reflective modular form for $2U\oplus 2A_4$ can be viewed as a $10$-reflective modular form for $2U(5)\oplus 2A_4^\vee(5)$. By \cite{Bru14}, this $10$-reflective modular form is a Borcherds product. This contradicts \cite[Proposition 6.1]{Sch17} because $10> 2 + 24/(5+1)$.

\smallskip
\textbf{(4) h=6:} The irreducible root lattices of Coxeter number $6$ are $A_5$ and $D_4$. By \S \ref{sec:quasipullback}, $M$ is $2$-reflective if $L=A_5$, $D_4$, $2D_4$. We claim that $M$ is not $2$-reflective for $L=2A_5$ and $L=A_5\oplus D_4$.  

Firstly, there is no $2$-reflective vector of norm $1/2$ ($\m 2$) (this type of vectors should have order $2$) in the discriminant group of $2A_5$. If there exists a $2$-reflective modular form for $2U\oplus 2A_5$, then it is a modular form with complete $2$-divisor.  But we have proved in \cite[Theorem 3.4]{Wan18} that if $M$ has a modular form with complete $2$-divisor then $\rank(L)\leq 8$. This gives a contradiction. 

Secondly, we have seen in \S \ref{sec:quasipullback} that the quasi pull-back of $A_5\oplus D_4 < N(4A_5\oplus D_4)$ gives a reflective modular form. As the proof of Theorem \ref{th:sign19}, we can check that this modular form has additional reflective divisor $\cH([2]\oplus [2], -1/6)$ given by the pull-backs of vectors of norm $4$ and type $[5]\oplus [2]\oplus [1]\oplus [0]\oplus [2]$. We denote the corresponding Jacobi form of weight $0$ by $\phi_1$. Suppose that $2U\oplus A_5\oplus D_4$ is $2$-reflective and we denote the corresponding Jacobi form of weight $0$ by $\phi_2$. Then $\phi_1$ and $\phi_2$ have the same $q^{-1}$ and $q^0$-terms. Moreover, $\phi:=\phi_1-\phi_2$ is a weakly holomorphic Jacobi form of weight $0$ whose singular Fourier coefficients  in the $q^1$-term are represented by $[3]\oplus [2]$ and  $[2]\oplus [2]$. Thus the minimum hyperbolic norm of singular Fourier coefficients of $\phi$ is $-1/2$. Therefore, $\eta^{6}\phi $ is a holomorphic Jacobi form of weight $3$ with a character. In view of the singular weight, this leads to a contradiction.

\smallskip
\textbf{(5) h=7:} The unique irreducible root lattice of Coxeter number $7$ is $A_6$. The lattice $M$ is $2$-reflective if $L=A_6$. 

\smallskip
\textbf{(6) h=8:} The irreducible root lattices of Coxeter number $8$ are $A_7$ and $D_5$. By \S \ref{sec:quasipullback}, $M$ is $2$-reflective if $L=A_7$, $D_5$. We claim that $M$ is not $2$-reflective for $L=2D_5$. This claim can be proved as the case of $2A_5$.

\smallskip
\textbf{(7) h=9:} The unique irreducible root lattice of Coxeter number $9$ is $A_8$. The lattice $M$ is not $2$-reflective if $L=A_8$.  We can prove this claim in a similar way as the case of $A_5\oplus D_4$.

\smallskip
\textbf{(8) h=10:} The irreducible root lattices of Coxeter number $10$ are $A_9$ and $D_6$. By \S \ref{sec:quasipullback}, the lattice $M$ is $2$-reflective if $L=D_6$. We can prove that $M$ is not $2$-reflective for $L=A_9$ in a similar way as the cases of $A_8$ and $A_5\oplus D_4$.

\smallskip
\textbf{(9) h=11:} The unique irreducible root lattice of Coxeter number $11$ is $A_{10}$. The lattice $M$ is not $2$-reflective if $L=A_{10}$. Since $A_{10}$ is of prime level $11$, if $2U\oplus A_{10}(-1)$ is 2-reflective then the corresponding 2-reflective modular form is a modular form with complete 2-divisor. This leads to a contradiction.

\smallskip
\textbf{(10) h=12:} The irreducible root lattices of Coxeter number $12$ are $A_{11}$, $E_6$ and $D_7$. 
The lattice $M$ is $2$-reflective if $L=D_7$, $E_6$. The lattice $M$ is not $2$-reflective for $L=A_{11}$, which can be proved as the case of $2A_5$.

\smallskip
\textbf{(11) h is larger than 12:} In view of $\rank(L)\leq 11$, the rest cases are $L=D_{m}$ with $8\leq m \leq 11$, $E_7$, $E_8$. The lattice $M$ is $2$-reflective if $L=D_8$, $D_{10}$, $E_7$, $E_8$. The lattice $M$ is not $2$-reflective for $L=D_{9}$, $D_{11}$, which can be proved as the case of $2A_5$.

\smallskip
\textbf{(12) The case of overlattices:} Let $L_1$ be a nontrivial even overlattice of $R(L)$ whose root sublattice generated by 2-roots is $R(L)$. In this case, the minimum norm of vectors in any nontrivial class of $L_1/R(L)$ is an even integer larger than $2$. It is easy to show that there is no such $R(L)$.

\smallskip
By the discussions above, we have thus proved the following theorem.

\begin{theorem}\label{th:2reflectNiemeier}
Let $M=2U\oplus L(-1)$ be a $2$-reflective lattice of Niemeier type. Then $L$ can only take one of the following $21$ lattices up to genus
\begin{align*}
&3E_8& &2E_8& &E_8& &E_7& &E_6& &A_2& &2A_2& &3A_2& &A_3& &2A_3& &A_4& \\ 
&A_5& &A_6& &A_7& &D_4& &2D_4& &D_5& &D_6& &D_7& &D_8& &D_{10}.&
\end{align*}
\end{theorem}

\subsection{Classification of 2-reflective lattices of quasi-Niemeier type}
In this subsection we classify $2$-reflective lattices of quasi-Niemeier type. We use the notations in Theorem \ref{th:2-reflective}. Let $R(L)=mA_1\oplus R$ and $L=mA_1\oplus L_0$. We assume $\rank(L)\leq 11$. Let $M=2U\oplus L(-1)$. By Lemma \ref{Lem:pullbacktest}, if $M$ is $2$-reflective then $M=2U\oplus L_0(-1)\oplus (m-1)A_1(-1)$ is also $2$-reflective. Therefore, we only need to consider the root lattices given in Theorem \ref{th:2reflectNiemeier} for $L_0$. 
We shall prove the following theorem.

\begin{theorem}\label{th:quasiNiemeier}
Let $M=2U\oplus L(-1)$ be a $2$-reflective lattice of quasi-Niemeier type. Then $L$ is in  the genus of one of the following $21$ lattices 
\begin{align*}
A_1 &\oplus \{E_8, E_7, E_6, A_2, 2A_2, A_3, A_4, A_5, D_4, 2D_4, D_5, D_6, D_8\}\\
2A_1 & \oplus \{E_8, A_2, A_3, D_4, D_6\}\\
D_4& \oplus \{3A_1, 4A_1, 5A_1\} .
\end{align*}
Note that $3A_1\oplus D_6$ and $A_1\oplus 2D_4$, $2A_1\oplus E_7$ and $A_1\oplus D_8$ have the isomorphic discriminant form respectively.
\end{theorem}

\begin{proof}
By \S \ref{sec:quasipullback}, when $L$ takes one of the above lattices,
the lattice $M$ is $2$-reflective. We next prove that $M$ is not $2$-reflective for other lattices.

\smallskip
\textbf{(1)} Since
$4A_1<D_4$ and  $6A_1<D_6$,   
we have $m\leq 5$. In addition, when $m=4$ or $5$, $L_0=D_4$ or $A_5$. But when $L_0=A_5$, $4A_1\oplus A_5 < A_5\oplus D_4$, which is impossible because $2U\oplus A_5\oplus D_4$ is not $2$-reflective. Thus when $m\geq 4$, the lattice $L_0$ can only take $D_4$.

\smallskip
\textbf{(2)} The lattice $L$ is not equal to $2D_4\oplus mA_1$ for $m\geq 2$ because $2D_4\oplus 2A_1 < D_4\oplus D_6$.

\smallskip
\textbf{(3)} The lattice $L$ is not equal to $E_8\oplus 3A_1$. If $2U\oplus E_8\oplus 3A_1$ is $2$-reflective, then we have by Theorem \ref{th:2-reflective} that $\beta_1=57\beta_0$ and $k=81\beta_0$. By Theorem \ref{th:BorcherdsJF}, the $q^0$-term of the corresponding weakly holomorphic Jacobi form of weight $0$ would define a holomorphic Jacobi form for $E_8\oplus 3A_1$ as a theta block (see \eqref{FJtheta}).  Thus the function corresponding to each copy of $A_1$
$$
\eta^{(162-8)/3} \left( \frac{\vartheta(\tau,2z)}{\eta(\tau)} \right)\left( \frac{\vartheta(\tau,z)}{\eta(\tau)} \right)^{56}
$$
is a holomorphic Jacobi form of index $30$ for $A_1$. We calculate the hyperbolic norm of its first Fourier coefficient
$$
4\times \frac{1}{24}\left( \frac{154}{3}+57\times 2 \right)\times 30 - 29^2 = -14.333... < 0,
$$
which contradicts the definition of holomorphic Jacobi forms.

Since $D_{10}$ and $E_8\oplus 2A_1$ are in the same genus, $M$ is not $2$-reflective if $L=D_{10}\oplus A_1$. 

\smallskip
\textbf{(4)} $L\neq D_{8}\oplus mA_1$ for $m\geq 2$. It is because that $D_8\oplus 2A_1$ and $D_4\oplus D_6$ are in the same genus. Furthermore, $L\neq E_{7}\oplus 3A_1$ because $U\oplus E_7\oplus 3A_1 \cong U\oplus D_8\oplus 2A_1$.

\smallskip
\textbf{(5)} $L\neq A_1 \oplus 3A_2$. Otherwise, suppose there exists a $2$-reflective modular form for $2U\oplus A_1\oplus 3A_2$ and we note the corresponding Jacobi form of weight $0$ by $\phi$. On the other hand, the pull-back of $\varphi_{0,N(R)}$ on $A_1\oplus 3A_2 < N(12A_2)$ will also give a Jacobi form of weight $0$ which is noted by $\phi_1$. Using the idea in this section, we conclude that $\phi$ and $\phi_1$ have the same $q^0$-term and the difference $\psi:=\phi-\phi_1$ would give a Jacobi form of weight $0$ and index $A_1\oplus 3A_2$ without $q^{-1}$ and $q^0$-terms. This function is not zero and its singular Fourier coefficients are represented by $(\frac{1}{2})\oplus [1]^3$ which has hyperbolic norm $-1/2$ and does not correspond to $2$-reflective divisors. Thus $\eta^{6}\psi$ is a holomorphic Jacobi form of weight $3$ with a character for $A_1\oplus 3A_2$, which contradicts the singular weight.

$L\neq 3A_1\oplus A_3$. Suppose that $2U\oplus 3A_1\oplus A_3$ is $2$-reflective and we denote the corresponding Jacobi form of weight $0$ by $\phi$. The pull-back of $\varphi_{0,N(R)}$ on $3A_1\oplus A_3 < N(8A_3)$ gives a Jacobi form of weight $0$ (noted by $\phi_1$). The functions $\phi$ and $\phi_1$ have the same $q^0$-term and their difference $f:=\phi-\phi_1=O(q)$ is a Jacobi form of weight $0$ for $3A_1\oplus A_3$. This function is not zero and its singular Fourier coefficients are represented by $v_1:=(\frac{1}{2},\frac{1}{2},\frac{1}{2})\oplus [1]$ (with hyperbolic norm $-1/4$) and $v_2:=(\frac{1}{2},\frac{1}{2},\frac{1}{2})\oplus [2]$ (with hyperbolic norm $-1/2$). Hence $\eta^{6}f$ is a holomorphic Jacobi form of singular weight $3$ with a character for $3A_1\oplus A_3$. This contradicts the singular weight because there is a non-zero Fourier coefficient with non-zero hyperbolic norm, i.e. $q^{1/4}\zeta^{(v_1,\mathfrak{z})}$ with hyperbolic norm $1/4$.

All other cases can be proved in a similar way. Since the pull-back of $\varphi_{0,N(R)}$ has additional singular Fourier coefficients in its $q^1$-term which do not correspond to $2$-reflective divisors, we can construct a holomorphic Jacobi form of low weight with a character, which would contradict the singular weight.  This completes the proof.
\end{proof}

\subsection{Classification of 2-reflective lattices of other type}\label{sec:othertype}
In this subsection, we discuss the final case: $R(L)=nA_1$. Firstly, if $\beta_0=0$, the only possible case is $L=nA_1$. In this case, the weight $k$ is equal to $(6-n)\beta_1$. In view of the singular weight, we have $k\geq n\beta_1/2$ since $\eta^{2k/n}(\vartheta(\tau,z)/ \eta)^{\beta_1}$ is a holomorphic Jacobi form. Therefore we get $1\leq n \leq 4$. The corresponding $2$-reflective modular forms can be constructed as the quasi pull-backs of the $2$-reflective modular form of singular weight $2$ for $2U\oplus 4A_1$ (see \cite[\S 5.1]{Gri18}). In view of Theorem \ref{th:2-reflective}, we thus prove the following.

\begin{theorem}
If $M=2U\oplus L(-1)$ has a $2$-reflective modular form with $\beta_0=0$ in its zero divisor, then $L=nA_1$ with $1\leq n \leq 4$.
\end{theorem}

By \S \ref{sec:quasipullback}, when $L=nA_1$ with $1\leq n \leq 8$, the lattice $M$ is $2$-reflective. For the overlattices, $2U\oplus N_8$ and $2U\oplus N_8\oplus A_1$ are $2$-reflective. The lattice $2U\oplus N_8\oplus 2A_1$ is not $2$-reflective because
$$
2A_1\oplus N_8 < 2A_1\oplus 2D_4 < D_4\oplus D_6.
$$

To complete the classification, we show that $2U\oplus 9A_1$ is not $2$-reflective. Conversely, suppose that there exists a $2$-reflective modular form for $2U\oplus 9A_1$. By \S \ref{sec:quasipullback}, the lattice $2U\oplus 8A_1$ is $2$-reflective and the $2$-reflective modular form is constructed as the quasi pull-back on $8A_1<N(8A_3)$. For this $2$-reflective modular form, we have $\beta_1=5$. We claim that this function is the unique $2$-reflective modular form for $2U\oplus 8A_1$ up to a constant. Otherwise, by considering the difference between the two independent $2$-reflective modular forms,  we would get a weak Jacobi form of weight $0$ for $8A_1$ whose minimal hyperbolic norm of singular Fourier coefficients is $-1/2$. Thus its product with $\eta^6$ would give a holomorphic Jacobi form of weight $3$ for $8A_1$, which contradicts the singular weight. 

The quasi pull-back of the $2$-reflective modular form for $2U\oplus 9A_1$ is the $2$-reflective modular form for $2U\oplus 8A_1$. Therefore, we have $\beta_1=5\beta_0$ in the case of $9A_1$. Thus the weight is given by
$$
k=\beta_0\left( 12+18\left(\frac{12}{9}-\frac{1}{2}\right) \right)+\left(\frac{3}{9}-\frac{1}{2}\right)\times 18\times 4\beta_0=15\beta_0.
$$
The $q^0$-term of the corresponding Jacobi form of weight $0$ defines a holomorphic Jacobi form for $9A_1$ as a theta block. Then the part related to each copy of $A_1$
$$
\eta^{30/9}(\tau) \left( \frac{\vartheta(\tau,2z)}{\eta(\tau)} \right)\left( \frac{\vartheta(\tau,z)}{\eta(\tau)} \right)^4 
$$
is a holomorphic Jacobi form of index $4$ for $A_1$. But the hyperbolic norm of its first Fourier coefficient is 
$$
4\times \frac{1}{24}\left( \frac{30}{9} + 2\times 5 \right)\times 4 - 3^2= - \frac{1}{9}<0,
$$
which gives a contradiction.

\subsection{Final classification}

\begin{proof}[Proof of Theorem \ref{th:non12}]
Combining \cite[Theorem 3.8]{Wan18}, Theorem \ref{th:nonsign12}, Theorem \ref{th:sign19} together, we complete the proof.
\end{proof}

\begin{proof}[Proof of Theorem \ref{th:main2reflective}]
By Theorem \ref{th:2-reflective}, Theorem \ref{th:sign19}, Theorem \ref{th:2reflectNiemeier}, Theorem \ref{th:quasiNiemeier} and \S \ref{sec:othertype}, we conclude the proof. The only thing that we need to explain is the following. Every lattice listed in (c) has a $2$-reflective modular form constructed as a quasi pull-back of $\Phi_{12}$. For every such quasi pull-back, we have that $\beta_0=1$ and the weight $k>12$. Thus every lattice in the genus of $L$ has $2$-roots. Moreover, the quasi pull-back is a cusp form and then its Weyl vector has positive norm. 
\end{proof}

\begin{proof}[Proof of Corollary \ref{th:2reflectivesingularweight}]
It is a direct consequence of Theorem \ref{th:main2reflective} and the weight formula \eqref{eq:weight}.
\end{proof}

As an application, we give a classification of modular forms with complete $2$-divisor.

\begin{theorem}\label{th:complete2reflective}
Suppose $M=2U\oplus L(-1)$ has a modular form with complete $2$-divisor and the set of $2$-roots of $L$ is non-empty. Then $L$ is in the genus of $3E_8$ or one of the lattices formulated in Table \ref{tab: GN1}.
\end{theorem}
\begin{proof}
Firstly, from the formula $\beta_1=(2h-3)\beta_0$ in Theorem \ref{th:2-reflective}, we see that there is no modular form with complete $2$-divisor for $2$-reflective lattices of quasi-Niemeier type. We next consider the lattices of type $mA_1$. We first construct a $2$-reflective modular form for $2U\oplus 5A_1$ whose $2$-reflective divisor of type $(0,1,(\frac{1}{2},\frac{1}{2},\frac{1}{2},\frac{1}{2},\frac{1}{2}),1,0)$ has multiplicity $9$.   Let $8A_1=\oplus _{i=1}^8\ZZ\alpha_i $ with $\alpha_i^2=2$. The Nikulin lattice is $N_8=\latt{8A_1, h}$, where $h=\frac{1}{2}\sum_{i=1}^8 \alpha_i$. It is known that there is a modular form with complete $2$-divisor for $2U\oplus N_8$ (see \S \ref{sec:quasipullback}). Thus there is a weakly holomorphic Jacobi form $\phi_{0,N_8}$ of weight $0$ for $N_8$ with the singular Fourier coefficients
$$
\sing(\phi_{0,N_8})=q^{-1}+56+\sum_{n\in \NN}\sum_{\substack{r\in N_8\\ r^2=2n+2}} q^n e^{2\pi i (\mathfrak{z},r)}.
$$
We consider the pull-back on $5A_1<N_8$
$$
\phi_{0,5A_1}(\tau,\mathfrak{z}_5)=q^{-1}+62+\sum_{i=1}^5 \zeta_i^{\pm 2} +O(q),
$$ 
where $\mathfrak{z}_5=\sum_{i=1}^5z_i\alpha_i$ and $\zeta_i=e^{2\pi iz_i}$.
We need to determine the singular Fourier coefficients in the $q^1$-term. This type of Fourier coefficients is of the form $\frac{1}{2}\sum_{i=1}^5 \alpha_i$ and it comes from the pull-back of vectors of norm $4$ in $N_8$  of type $\frac{1}{2}\sum_{i=1}^5 \alpha_i \pm \frac{1}{2}\alpha_6\pm \frac{1}{2}\alpha_7\pm \frac{1}{2}\alpha_8$. Thus the coefficient of $q\zeta_1\zeta_2\zeta_3\zeta_4\zeta_5$ is $8$. The Borcherds product of $\phi_{0,5A_1}$ gives the desired $2$-reflective modular form.  Suppose that $2U\oplus 5A_1$ has a modular form with complete $2$-divisor and we denote the corresponding Jacobi form of weight zero by $\psi_{0,5A_1}$. Then $g:=\phi_{0,5A_1}-\psi_{0,5A_1}$ is a non-zero weak Jacobi form of weight $0$ without $q^0$-term. It follows that $g/\Delta$ is a weak Jacobi form of weight $-12$, which is impossible because the minimal weight of weak Jacobi forms of index $1$ for $5A_1$ is $-10$ (see \cite[Theorem 3.6]{Wir92} or \cite[Theorem 3.1]{Wan21}). Thus, $2U\oplus 5A_1$ has no modular form with complete $2$-divisor. In view of the pull-back, we conclude that $2U\oplus mA_1$ has no modular form with complete $2$-divisor when $m\geq 6$. The proof is complete.
\end{proof}

Note that there are in fact two independent 2-reflective modular forms for $2U\oplus 5A_1$. The second one can be constructed as the quasi pull-back of $\Phi_{12}$ on $5A_1<N(8A_3)$ (see \S \ref{sec:quasipullback}).

Remark that there are lattices not of type $2U\oplus L$ which have a modular form with complete $2$-divisor, such as $U(2)\oplus \latt{-2}\oplus (k+1)\latt{2}$ with $1\leq k\leq 7$ (see \cite[Theorem 6.1]{GN18}).

\begin{proposition}
If $M$ is a maximal even lattice of signature $(2,10)$ having a modular form with complete $2$-divisor, then it is isomorphic to $\II_{2,10}$.
\end{proposition}

\begin{proof}
It is a refinement of the proof of \cite[Theorem 3.4]{Wan18}. Firstly, the lattice $M$ can be written as $M=2U\oplus L(-1)$. There exists a weakly holomorphic Jacobi form of weight $0$ and index $L$. As the proof of \cite[Theorem 3.4]{Wan18}, we can construct a holomorphic Jacobi form of weight $4$ and index $L$, denoted by $g$. It is easy to check that the constant term of $g$ is not zero and we assume it to be $1$. The function $g$ has singular weight $4$. Thus, it is a $\CC$-linear combination of theta-functions for $L$ defined as \eqref{eq:ThetaFunction}. Since $L$ is maximal, there is no $\gamma\in L^\vee$ such that $\gamma\not\in L$ and $(\gamma,\gamma)=2$. Hence, the $q^1$-term of the Fourier expansion of $g$ comes only from the theta-function $\Theta_0^L$. In view of $g(\tau,0)=E_4(\tau)=1+240q+...$, the number of $2$-roots in $L$ is $240$. By Theorem \ref{th:2-reflective}, the Coxeter number of $L$ is $30$, which forces $L$ to be isomorphic to $E_8$. The proof is complete.
\end{proof}

\section{Application: automorphic correction of hyperbolic 2-reflective lattices}\label{sec:autocorrection}
An even lattice $S$ of signature $(1,n)$ is called hyperbolic 2-reflective if the subgroup generated by $2$-reflections is of finite index in the orthogonal group of $S$, i.e.
$$
W^{(2)}=\latt{\sigma_r: r\in S, r^2=-2}< \Orth(S)
$$
is of finite index. The lattice $S$ is called hyperbolic reflective if the subgroup generated by all reflections has finite index in $\Orth(S)$.  Hyperbolic reflective lattices are closely related to reflective modular forms. In \cite[Theorem 12.1]{Bor98}, Borcherds proved that if the lattice $U\oplus S$ has a reflective (resp. $2$-reflective) modular form with a Weyl vector of positive norm then $S$ is hyperbolic reflective (resp. $2$-reflective).

Hyperbolic 2-reflective lattices are of a special interest because of its close connection with the theory of K3 surfaces.  The classification of such lattices is now available thanks to the work of Nikulin and Vinberg (see \cite{Nik81} for $n\geq 4$, \cite{Nik84} for $n=2$, \cite{Vin07} for $n=3$, and a survey \cite{Bel16}). Table \ref{tab:hyperbolic} gives the number of hyperbolic 2-reflective lattices of fixed rank. The models of all these lattices can be found in \cite[\S 3.2]{GN18}. For $\rank(S)=10$, we need to add the lattice $U\oplus D_4\oplus 4A_1$ to the table in \cite[\S 3.2]{GN18}.

\begin{table}[ht]
\caption{The number of hyperbolic 2-reflective lattices}
\label{tab:hyperbolic}
\renewcommand\arraystretch{1.5}
\noindent\[
\begin{array}{|c|c|c|c|c|c|c|c|c|c|c|c|c|c|c|}
\hline 
\rank(S) & 3& 4& 5& 6& 7& 8& 9& 10& 11& 12& 13& 14& 15,...,19& \geq 20 \\ 
\hline 
\text{Number}& 26& 14& 9& 10& 9& 12& 10& 9& 4& 4& 3& 3& 1&  0\\
\hline
\end{array} 
\]
\end{table}

In \cite{Bor00}, Borcherds suggested that interesting hyperbolic reflective lattices should be associated to reflective modular forms. In view of this suggestion, Gritsenko and Nikulin \cite{GN18} considered the following automorphic correction of hyperbolic $2$-reflective lattices. 

\begin{definition}
Let $S$ be a hyperbolic $2$-reflective lattice. If there exists a positive integer $m$ such that $U(m)\oplus S$ has a $2$-reflective modular form, then we say that $S$ has an automorphic correction. 
\end{definition}

By means of the classification results in the previous section, we prove the following theorem.

\begin{theorem}\label{th:autocorrection}
Let $S$ be a hyperbolic $2$-reflective lattice of signature $(1,n)$ with $n\geq 5$. If $S$ is one of the following $18$ lattices 
\begin{align*}
&U\oplus E_8\oplus E_7& &U\oplus E_8\oplus D_6& &U\oplus E_8\oplus D_4\oplus A_1& &U\oplus E_8\oplus D_4& &U\oplus D_8\oplus D_4&\\
&U\oplus E_8\oplus 4A_1& &U\oplus E_8\oplus 3A_1& &U\oplus D_8\oplus 3A_1& &U\oplus E_8\oplus A_3& &U\oplus D_8\oplus 2A_1&\\
&U\oplus 2D_4\oplus 2A_1& &U\oplus E_8\oplus A_2& &U\oplus E_6\oplus A_2& &U\oplus D_4\oplus A_3& &U\oplus D_5\oplus A_2&\\
&U\oplus D_4\oplus A_2& &U\oplus A_4\oplus A_2& &U\oplus A_3\oplus A_2&
\end{align*}
then it has no automorphic correction. If $S$ is one of the other $51$ lattices, it has at least one automorphic correction.
\end{theorem}

\begin{proof}
If $U(m)\oplus S$ is $2$-reflective, then $U\oplus S$ is also $2$-reflective. We then prove the result by Theorem \ref{th:2-reflective} (b) and (c). The automorphic corrections of $S$ can be found in \S \ref{sec:quasipullback} and \cite{GN18}.
\end{proof}

For $2\leq n \leq 4$, there are a lot of hyperbolic 2-reflective lattices not of type $U\oplus L(-1)$. Our argument does not work well in this case.

\begin{remark}
It is possible to use the classification of hyperbolic $2$-reflective lattices to prove Theorem \ref{th:main2reflective}. The Weyl vector of a Borcherds product is given by $(A,\vec{B},C)$ in Theorem \ref{th:BorcherdsJF}. For a $2$-reflective modular form with respect to the lattice of type $2U\oplus mA_1\oplus L_0$ with $m\geq 0$ and $L_0\neq \emptyset$ (see notations in Theorem \ref{th:2-reflective}), we have 
$$
(A,\vec{B}, C)=\left( h+1, \sum \rho_i + \frac{h-1}{2}\sum \alpha_j, h  \right),
$$
where $\rho_i$ is the Weyl vector of the irreducible components of the root sublattice of $L_0$ and $\alpha_j$ are the positive roots of $mA_1$.  We thus calculate the norm of the Weyl vector as
$$
2AC-(\vec{B},\vec{B})=\frac{h(h+1)}{12} \left( 24- n-5m + \frac{6m(3h-1)}{h(h+1)} \right).
$$
If the Weyl vector has positive norm, then the lattice $U\oplus mA_1\oplus L_0$ is hyperbolic $2$-reflective. 
We can show that for almost all lattices determined by Theorem \ref{th:2-reflective} the norm of Weyl vectors are positive. For example, when $m=0$ and $n<24$, we have $2AC-(\vec{B},\vec{B})=\frac{h(h+1)}{12}(24-n)>0$. Thus they are all hyperbolic $2$-reflective and we may use the classification of hyperbolic $2$-reflective lattices to determine $2$-reflective lattices.
\end{remark}

\begin{remark}\label{rem:pullback}
Let $L$ be a primitive sublattice of a Niemeier lattice $N(R)$. If the orthogonal complement of $L$ on $N(R)$ has 2-roots, then every reflective modular form for $2U\oplus L(-1)$ constructed as the quasi pull-back of $\Phi_{12}$ is a cusp form (see Theorem \ref{th: Borchpullback}) and then has a Weyl vector of positive norm, which yields that the corresponding Lorentzian lattice is hyperbolic reflective.

Note that the sublattices $6A_2< N(6D_4)$ and $12A_1<N(12A_2)$ do not satisfy the above assumption. By direct calculations, the Weyl vectors of the corresponding reflective modular forms have zero norm.

It is now easy to see that the lattice $U\oplus L(-1)$ are hyperbolic reflective for some $L$ in \S \ref{sec:quasipullback}, such as $L=2E_8\oplus D_4$, $2E_8\oplus 2A_1$, $2E_8\oplus A_1(2)$, $E_8\oplus D_9$, $E_8\oplus 2D_4$, $E_8\oplus D_7$, $2E_7$, $E_8\oplus D_4\oplus A_1(2)$, $11A_1$, $5A_2$, $A_5\oplus D_4$, $D_5\oplus A_2$, and so on.
\end{remark}

\section{Application: automorphic discriminants of moduli spaces of K3 surfaces}\label{sec:K3surfaces}
The moduli space of polarized K3 surfaces of degree $2n$ can be realized as the modular variety $\widetilde{\Orth}^{+}(T_n)\backslash\cD(T_n)$, where 
\begin{equation}
T_n=U\oplus U \oplus E_8(-1)\oplus E_8(-1)\oplus \latt{-2n}
\end{equation}
is an even lattice  of signature $(2,19)$. The discriminant of this moduli space is equal to the $(-2)$-Heegner divisor $\cH$. Nikulin \cite{Nik96} asked the question if the discriminant can be given by the set of zeros of some automorphic form. This question is equivalent to whether $T_n$ is 2-reflective.  Looijenga \cite{Loo03} gave the answer  that $T_n$ is not 2-reflective if $n\geq 2$.  Now, this result is immediately derived from Theorem \ref{th:2-reflective} because the set of $2$-roots of $2E_8\oplus \latt{2n}$ does not span the whole space $\RR^{17}$ when $n\geq 2$. Moreover, Theorem \ref{th:sign19} gives a generalization of this result. Nikulin \cite{Nik96} also asked the similar question for more general lattice-polarized K3 surfaces. Theorem \ref{th:non12} implies the nonexistence of such good automorphic forms for other large rank lattices. Many $2$-reflective modular forms for small rank lattices related to lattice-polarized K3 surfaces were constructed in \cite{GN17}.

As another application of our approach, we further prove the following result.

\begin{theorem}\label{th:reflective19}
The lattice $T_n$ is reflective if and only if $n=1$, $2$.
\end{theorem}

\begin{proof}
We have seen from \S \ref{sec:quasipullback} that $T_1$ and $T_2$  are reflective.  We next suppose that $n\geq 3$ and $T_n$ is reflective. Then there exists a weakly holomorphic Jacobi form of weight $0$ and index $1$ for $2E_8\oplus \latt{2n}$ with first Fourier coefficients of the form
$$
\phi(\tau,\mathfrak{z})=c_0 q^{-1} + c_0\sum_{\substack{r\in 2E_8\\ r^2=2}} e^{2\pi i (\mathfrak{z},r)} + c_1\zeta^{\pm \frac{1}{n}} + c_2\zeta^{\pm \frac{1}{2n}} + 2k + O(q),
$$
where $c_0$, $c_1\in \NN$, $c_2\in \ZZ$ satisfying $c_1+c_2\geq 0$, $\zeta^{\pm \frac{1}{n}}=\exp(2\pi i (\mathfrak{z}, \pm\frac{1}{n}\alpha))$, $\zeta^{\pm \frac{1}{2n}}=\exp(2\pi i (\mathfrak{z}, \pm\frac{1}{2n}\alpha))$, $\alpha$ is the basis of the lattice $\latt{2n}$ with $\alpha^2=2n$. The reason we have $c_1+c_2\geq 0$ is that it is the multiplicity of the Heegner divisor $\cH(\frac{1}{2n}\alpha, -\frac{1}{4n})$.
By Lemma \ref{Lem:q^0-term}, we get
\begin{align*}
60nc_0&=4c_1+c_2,\\
k+c_1+c_2&=132c_0.
\end{align*}
We can assume $c_0=1$. The $q^0$-term of $\phi$ defines a holomorphic Jacobi form for $2E_8\oplus \latt{2n}$ as a generalized theta block. In particular, the part related to $\latt{2n}$
$$
\eta^{2k-16}(\tau) \left( \frac{\vartheta(\tau,2z)}{\eta(\tau)} \right)^{c_1}\left( \frac{\vartheta(\tau,z)}{\eta(\tau)} \right)^{c_2}
$$
is a holomorphic Jacobi form of index $30n$ for $A_1$. Thus, the hyperbolic norm of its first Fourier coefficient should be non-negative. We calculate it as
\begin{align*}
&4\times \frac{2k-16+2c_1+2c_2}{24}\times 30n - \left( \frac{2c_1+c_2}{2} \right)^2\\
=&1240n-\left( \frac{4c_1+c_2}{6}+ \frac{c_1+c_2}{3} \right)^2\\
\leq & 1240n-100n^2,
\end{align*}
which implies that $1240n\geq 100n^2$, i.e. $n\leq 12$. The last inequality follows from $c_1+c_2\geq 0$. 

If $\phi$ has no singular Fourier coefficients of hyperbolic norm $-1$, then as in \cite{Wan18}, by using the differential operators to kill the term $q^{-1}$ (consider a linear combination of $E_4\phi$ and $H_{2}(H_{0}(\phi))$), we can construct a non-zero weak Jacobi form $\phi_4$ of weight $4$ whose hyperbolic norms of singular Fourier coefficients are $>-1$, more precisely $\geq -2/3$ (see the description of reflective vectors in \S \ref{Sec:reflective}). Then $\eta^8\phi_4$ is a holomorphic Jacobi form of weight $8$ with a character for $2E_8\oplus \latt{2n}$, which contradicts the singular weight. 

Thus $\phi$ must have singular Fourier coefficients of hyperbolic norm $-1$. When $n\leq 12$, the singular Fourier coefficients of $\phi$ are determined by its $q^{-1}$, $q^0$, $q^1$ and $q^2$-terms. Since the singular Fourier coefficients of $\phi$ should correspond to reflective divisors, the singular Fourier coefficients of hyperbolic norm $-1$ are represented by $\frac{1}{2}\alpha$ with $\frac{\alpha^2}{4}=1 (\m 2)$ because the order must be $2$.  The only possible case is $n=6$ or $10$. 

In the case $n=6$, the possible singular Fourier coefficients of $\phi$ are: $q^{-1}$, $\zeta^{\pm 1/6}$, $\zeta^{\pm 1/12}$, and $q\zeta^{\pm 1/2}$ with hyperbolic norms $-2$, $-1/3$, $-1/12$, and $-1$, respectively.  Similarly, by using differential operators to kill the terms $q^{-1}$ and $q\zeta^{\pm 1/2}$ (consider a linear combination of $E_6\phi$, $E_4 H_{0}(\phi)$ and $H_4(H_2(H_{0}(\phi)))$), we can construct a non-zero weak Jacobi form $\phi_6$ of weight $6$ with only singular Fourier coefficients of types $\zeta^{\pm 1/6}$ and $\zeta^{\pm 1/12}$. Then $\eta^4\phi_6$ gives a  holomorphic Jacobi form of weight $8$ for $2E_8\oplus \latt{12}$, which contradicts the singular weight. Therefore, $T_6$ is not reflective. 
We can prove the case $n=10$ in a similar way. The proof is complete.
\end{proof}

\section{Further remarks and open questions}\label{sec:open}
We first give some remarks about the main results.

\begin{remark}
Similar to Theorem \ref{th:main2reflective}, replacing 2-roots with root system $\mathfrak{R}_L$ (see Proposition \ref{prop:reflectivetype}), we find that there are also exactly three types of reflective lattices containing $2U$. But $\II_{2,26}$ is not the unique reflective lattice of type $(a)$. In fact, 
$$
U\oplus \text{Coxeter-Todd lattice}\cong U\oplus 6A_2 \cong U \oplus E_6\oplus E_6^\vee(3)
$$
is also a reflective lattice of type (a). Besides, the reflective lattice of type (c) may have no a reflective modular form with a positive-norm Weyl vector, for example, $2U\oplus A_6^\vee(7)$ has a unique reflective modular form and this modular form has singular weight $3$ (see \cite{GW19}). Thus, the case of reflective is different from the case of $2$-reflective. We remark that every lattice having a reflective modular form of singular weight is of type (c) except $\II_{2,26}$.
\end{remark}

\begin{remark}
Our Jacobi forms approach is also useful to study the genus of a certain lattice because one has different Jacobi forms for different lattices in some genus.
From Theorem \ref{th:main2reflective} and its proof, we conclude 
\begin{enumerate}
\item The genus of $2E_8\oplus A_1$ contains exactly $4$ lattices: itself, the unique nontrivial even overlattices of $D_{16}\oplus A_1$, $A_{17}$, and $D_{10}\oplus E_7$. 
\item For $L= 2E_8$, $5A_1\oplus D_4$, $A_1\oplus 2D_4$, $A_1\oplus D_8$, or $E_8\oplus 2A_1$,  the genus of $L$ contains exactly $2$ lattices. 
\item Let $L$ take one of the rest $44$ lattices in the table of Theorem \ref{th:main2reflective}. The genus of $L$ contains only one lattice.
\end{enumerate} 

\end{remark}

\begin{remark}
Theorem \ref{th:non12} holds for meromorphic $2$-reflective modular forms. Firstly, from its proof, we see that Theorem \ref{th:2-reflective} is still true for meromorphic $2$-reflective modular forms. Secondly, in the proofs of \cite[Theorem 3.6]{Wan18} and Theorem \ref{th:nonsign12}, we only need to make minor correction for the cases $\rank(L)=12, 13, 14$. In these cases, we need to show that the constant $u$ (determined by the weight) of holomorphic Jacobi form $\phi_6$ is not zero. This can be done using Theorem \ref{th:2-reflective}. 
\end{remark}

\begin{remark}
For any $2$-reflective lattice of type $2U\oplus L(-1)$ with $\rank(L)>6$, the corresponding $2$-reflective modular form is unique up to a constant multiple. Indeed, if there exist two independent $2$-reflective modular forms, then there would exist a weakly holomorphic Jacobi form $\psi$ of weight $0$ and index $L$ without $q^{-1}$-term. Then $\eta^6\psi$ would be a holomorphic Jacobi form of weight $3$ with a character, which contradicts the singular weight.

Similarly, for any reflective lattice of type $2U\oplus L(-1)$ with $\rank(L)>12$, the corresponding reflective modular form is unique up to a constant multiple.
\end{remark}

\begin{remark}
By \S \ref{subsec:4th argument}, $2U\oplus E_8(-1)\oplus D_7(-1)$ is reflective. Thus there exist reflective lattices of signature $(2,17)$. This answers one part of \cite[Questions 4.13 (2)]{Wan18}. 
\end{remark}

In \cite[Theorem 4.11]{Wan18}, we proved that $\II_{2,26}$ is the unique reflective lattice of signature $(2,n)$ with $n\geq 23$ up to scaling. In fact, we can also prove that  reflective lattices of signature $(2,22)$ satisfy a restrictive condition.

\begin{proposition}\label{prop:sign20}
If $M=2U\oplus L(-1)$ is a reflective lattice of signature $(2,22)$ and $F$ is a reflective modular form for $\widetilde{\Orth}^+(M)$, then the weight of $F$ is $24\beta_0$ and the divisor of $F$ is given by
\begin{equation}\label{eq:div20}
\div(F)=\beta_0 \cH + \sum_v \beta_v \cH(v,-1/2),
\end{equation}
where the sum takes over all elements of norm $-1 \mod 2$ and order $2$ in the discriminant group of $M$, $\beta_0$ and $\beta_v$ are natural numbers.
\end{proposition}

\begin{proof}
The proof is an improvement to the proof of \cite[Theorem 4.11]{Wan18}.
Let $\phi$ be the associated weakly holomorphic Jacobi form of weight $0$. Applying differential operators to $\phi$, we construct a weak Jacobi form $g$ of weight $4$ by canceling the singular Fourier coefficients of hyperbolic norm $-2$. 
If $F$ has other type of divisors, then the Jacobi form $g$ of weight 4 will have singular Fourier coefficients of hyperbolic norm $>-1$.  But $\eta^{12}g$ is a holomorphic Jacobi form of singular weight $10$ with a character. This contradicts the fact that the hyperbolic norm of non-zero Fourier coefficients of any holomorphic Jacobi form of singular weight is always zero. This also forces that the constant term of $g$ is zero, which yields  $k=24\beta_0$.
\end{proof}

\begin{corollary}
Assume that $2U\oplus L(-1)$ is a reflective lattice of signature $(2,22)$. Let $R_L$ denote the set of $2$-roots in $L$. We set
$$
R_1(L)=\{v\in L^\vee: (v,v)=1,\ 2v\in L  \}.
$$
Then $\abs{R_L}\geq 120$ and the set $R_L\cup R_1(L)$ generates the vector space  $\RR^{20}$. 
\end{corollary}
\begin{proof}
Suppose that $F$ is a reflective modular form for the lattice. Then $F$ has weight $24\beta_0$ and its divisor is of the form \eqref{eq:div20}. We denote the associated weakly holomorphic Jacobi form of weight $0$ and index $L$ by $\phi$. We define $R=\abs{R_L}$ and $R_1=\sum_{v\in R_1(L)} \beta_v$. By \eqref{eq:q^0-term}, we have
$$
\frac{1}{24}\left( \beta_0(R + 48) + R_1  \right) - \beta_0= \frac{1}{40}(2\beta_0R+R_1),
$$
which yields 
$$
R=120+\frac{2}{\beta_0}R_1.
$$
Then we have $R\geq 120$. The second assertion follows from \eqref{eq:vectorsystem}.
\end{proof}

In the end, we pose some interesting questions related to our work. Questions (1) and (2) have been formulated in \cite{Wan18}. But here, we further make conjectures and provide some evidence to support them.

\begin{enumerate}
\item Are there $2$-reflective lattices of signature $(2,13)$? 

By Theorem \ref{th:main2reflective}, there is no $2$-reflective lattice of signature $(2,13)$ which can be represented as $2U\oplus L(-1)$ satisfying that $L$ has $2$-roots.
If $M=2U\oplus L(-1)$ is a $2$-reflective lattice of signature $(2,13)$, then every lattice in the genus of $L$ has no 2-root. It is very possible that such $L$ does not exist. This suggests us that there might be no $2$-reflective lattice of signature $(2,13)$.

\item Are there reflective lattices of signature $(2,21)$? 

By \cite{Ess96}, there is no hyperbolic reflective lattice of signature $(1,20)$. In view of the relation between reflective modular forms and hyperbolic reflective lattices, we conjecture that the above question has a negative answer.

\item Let $2U\oplus L(-1)$ be a reflective lattice of signature $(2,22)$.  Is $L$ in the genus of $2E_8\oplus D_4$?

By \cite{Ess96}, $U\oplus 2E_8\oplus D_4$ is the unique maximal hyperbolic reflective lattice of signature $(1,21)$.
This result and Proposition \ref{prop:sign20} indicate that the answer may be positive.

\item Classify all $2$-reflective lattices of type $2U\oplus L(-1)$ satisfying that every lattice in the genus of $L$ has no $2$-roots. Corresponding to Theorem \ref{th:complete2reflective}, we conjecture that if the lattice $2U\oplus L(-1)$ has a modular form with complete $2$-divisor and $L$ has no 2-root then $L$ is a primitive sublattice of Leech lattice satisfying the $\norm_2$ condition.

\item Are Borcherds products of singular weight reflective? 

This question was mentioned in the introduction of \cite{Sch17}. At present, all known Borcherds products of singular weight are reflective except some pull-backs. For example, the Borcherds modular form $\Phi_{12}$ for $\II_{2,26}=2U\oplus 3E_8$ is a Borcherds product of 
$$
\frac{\Theta_{3E_8}(\tau,\mathfrak{z})}{\Delta(\tau)}=q^{-1}+24+\sum_{\substack{v\in 3E_8\\v^2=2}}e^{2\pi i (v,\mathfrak{z})}+O(q)\in J_{0,3E_8}^{!}.
$$
It is reflective and of singular weight. We consider the pull-back of $\frac{\Theta_{3E_8}}{\Delta}$ on the lattice $3D_8<3E_8$, which gives a Borcherds product of singular weight. By \cite[Theorem 4.7]{Wan18}, it is not reflective. Therefore, we suggest formulating the following conjecture.
\begin{conjecture}
Let $F$ be a Borcherds product of singular weight for a lattice $M$ of signature $(2,n)$ with $n\geq 3$. Then there exists an even lattice $M'$ such that $M'\otimes \QQ=M\otimes\QQ$ and $F$ can be viewed as a reflective modular form for $M'$. 
\end{conjecture}
The first step towards this conjecture was due to Scheithauer. It is known that $\Phi_{12}$ is the unique reflective modular form of singular weight for unimodular lattices. In \cite[Theorem 4.5]{Sch17}, Scheithauer proved that $\Phi_{12}$ is the unique Borcherds product of singular weight for unimodular lattices. This means that the above conjecture is true for unimodular lattices. Besides, in \cite{DHS15, OS18} the authors gave a classification of Borcherds products of singular weight on simple lattices. All Borcherds products of singular weight in their papers are reflective, which aslo supports the conjecture. 
\end{enumerate}

\bigskip

\noindent
\textbf{Acknowledgements} 
The work was substantially done when the author visited the Laboratory of Mirror Symmetry NRU HSE in October and November 2018. The author is grateful to the laboratory for its invitation and  hospitality. The author would like to thank Valery Gritsenko for helpful discussions and constant encouragement. The author is supported by the Institute for Basic Science (IBS-R003-D1). The author thanks the referee for careful reading and valuable comments.

\bibliographystyle{amsalpha}

\end{document}